\documentclass[a4paper,11pt, twoside]{amsart}

\usepackage{amsmath,amsfonts,amsthm,amsopn,xcolor,amssymb,enumitem}
\usepackage{palatino}
\usepackage[colorinlistoftodos]{todonotes}

\usepackage[colorlinks=true,urlcolor=blue,citecolor=red,linkcolor=blue,linktocpage,pdfpagelabels,bookmarksnumbered,bookmarksopen]{hyperref}
\hypersetup{urlcolor=blue, citecolor=red, linkcolor=blue}

\usepackage{bm}

\usepackage[italian,english]{babel}
\usepackage[left=2.61cm,right=2.61cm,top=2.72cm,bottom=2.72cm]{geometry}

\newcommand{\e}{\varepsilon}
\newcommand{\eps}{\varepsilon}

\newcommand{\R}{\mathbb{R}}

\newcommand{\RN}{{\mathbb{R}^N}}

\newcommand{\de}{\partial}

\renewcommand{\le}{\leslant}
\renewcommand{\ge}{\geslant}

\newcommand{\g }{\gamma }

\newcommand{\n }{\nabla }

\renewcommand{\O}{\Omega}

\newcommand{\G}{\Gamma}

\renewcommand{\H}{H^1(\RN)}
\newcommand{\Hr}{H^1_r(\RN)}

\renewcommand{\P}{\mathcal{P}}

\newcommand{\N}{\mathbb{N}}

\newcommand{\D }{{\mathcal D}^{1,2}(\RN)}

\newcommand{\irn }{\int_{\RN}}

\def\bbm[#1]{\mbox{\boldmath $#1$}}
\newcommand{\beq }{\begin{equation}}
\newcommand{\eeq }{\end{equation}}

\newcommand{\weakto}{\rightharpoonup}
\newcommand{\CS}{\mathcal{S}_{\rm rad}}

\renewcommand{\le}{\leqslant}
\renewcommand{\ge}{\geqslant}
\newcommand{\dis}{\displaystyle}

\newcommand{\Div}{\,\mathrm{div}}

\newcommand{\ef}{\eqref}

\newcommand{\quas}{\phi \left(\frac{u^2+|\n u|^2 }{2}\right)}
\newcommand{\Quas}{\Phi \left(\frac{u^2+|\n u|^2 }{2}\right)}
\newcommand{\quasd}{\phi' \left(\frac{u^2+|\n u|^2 }{2}\right)}

\newcommand{\quasn}{\phi \left(\frac{u_n^2+|\n u_n|^2 }{2}\right)}

\newcommand{\quasb}{\phi \left(\frac{s^2+b^2 }{2}\right)}
\newcommand{\Quasb}{\Phi \left(\frac{s^2+b^2 }{2}\right)}
\newcommand{\quasbd}{\phi' \left(\frac{s^2+b^2 }{2}\right)}
\newcommand{\quasbdd}{\phi'' \left(\frac{s^2+b^2 }{2}\right)}

\newcommand{\quasi}{\phi \left(\frac{u^2+|\bp|^2 }{2}\right)}

\newcommand{\quasid}{\phi' \left(\frac{u^2+|\bp|^2 }{2}\right)}
\newcommand{\quasidd}{\phi'' \left(\frac{u^2+|\bp|^2 }{2}\right)}

\newcommand{\bp}{\bm{p}}

\makeatletter
\providecommand\@dotsep{5}
\def\listtodoname{List of Todos}
\def\listoftodos{\@starttoc{tdo}\listtodoname}
\makeatother

\newtheorem{theorem}{Theorem}[section]
\newtheorem{lemma}[theorem]{Lemma}
\newtheorem{remark}[theorem]{Remark}
\newtheorem{proposition}[theorem]{Proposition}

\numberwithin{equation}{section}

\title
[
Quasilinear scalar field equations
]
{
Ground state solutions 
for quasilinear scalar field equations \\arising in nonlinear optics
}

\author[A. Pomponio]{Alessio Pomponio}
\author[T. Watanabe]{Tatsuya Watanabe}

\address[A. Pomponio]{\newline\indent
Dipartimento di Meccanica, Matematica e Management
\newline\indent 
Politecnico di Bari
\newline\indent
Via Orabona 4,  70125  Bari, Italy}
\email{alessio.pomponio@poliba.it}
\address[T. Watanabe]{\newline\indent 
Department of Mathematics, 
\newline\indent 
Faculty of Science, Kyoto Sangyo University,
\newline\indent
Motoyama, Kamigamo, Kita-ku, Kyoto-City, 603-8555, Japan}
\email{tatsuw@cc.kyoto-su.ac.jp}
\thanks{}

\subjclass[2010]{35J62, 35J20, 35Q60}
\date{}
\keywords{Quasilinear elliptic equation, variational method, monotone operator}

\begin{document}

\begin{abstract}
In this paper, we study a class of quasilinear elliptic equations 
which appears in nonlinear optics.
By using the mountain pass theorem together with a 
technique of adding one dimension of space \cite{HIT,jj}, 
we prove the existence of a non-trivial weak solution
for general nonlinear terms of Berestycki-Lions' type.
The existence of a radial ground state solution and 
a ground state solution is also established under
stronger assumptions on the quasilinear term.
\end{abstract}

\maketitle

\section{Introduction}
In this paper, we study the following quasilinear elliptic problem:
\begin{equation} \label{eq:1} 
\begin{cases}
\dis -\Div \left\{ \phi \left(\frac{u^2+|\n u|^2 }{2}\right) \n u\right\}
+\phi \left(\frac{u^2+|\n u|^2 }{2}\right)u = g(u) & \hbox{  in }\RN, \\
u(x)\to 0& \hbox{ as }|x|\to +\infty,
\end{cases}
\end{equation}
where $N\ge 3$. 
The purpose of this paper is to establish the existence of
nontrivial solutions and ground states solutions of \ef{eq:1} for general nonlinear term $g$
by using the variational method. 

On the map $\phi :[0,+\infty)\to \R$, we assume:
\begin{enumerate}[label=($\phi$\arabic{*}),ref=$\phi$\arabic{*}]
\item\label{phi1} $\phi \in C\big([0,+\infty)\big)$
and there exist two constants $0< \phi_0 <\phi_1$ such that
\[
\phi_0 \le \phi (s) \le \phi_1 \quad \hbox{for all} \ \ s \in [0,+\infty); 
\]
\item\label{phi2} the map $t \mapsto \Phi( \frac{t^2}{2} )$ 
is strictly convex on $\R$, where $\Phi(s)=\int_0^s \phi(\tau) \,d\tau$. 
\end{enumerate} 
We note that the condition \eqref{phi2} is equivalent to 
\[
\hbox{the map} \ t \mapsto t\phi(t^2) \ \hbox{is increasing on} \ [0,+\infty). 
\]
Typical examples of $\phi(s)$ are given by
\begin{itemize}
\item $\phi(s)=K+(1+s)^{-\alpha}$ with $0 \le \alpha \le \frac{1}{2}$, $K>0$ or $\alpha > \frac{1}{2}$, $K \gg 1$;

\item $\phi(s)=K-(1+s)^{-\alpha}$ with $\alpha \ge 0$, $K>1$;

\item $\phi(s)=K+(1+s)^{-\alpha}-(1+s)^{-\beta}$ with
$0 \le \alpha \le \frac{1}{2}$, $\beta \ge 0$, $K >1$ or
$\alpha > \frac{1}{2}$, $ \beta \ge 0$, $K \gg 1$;

\item $\phi(s)=K+(1+s)^{-\alpha}s^{\beta}$ with $0 \le \beta \le \alpha \le \beta + \frac{1}{2}$,
$K>0$ or $\beta \ge 0$, $\alpha > \beta + \frac{1}{2}$, $K \gg 1$;

\item $\phi(s)=K+\log\big(1+(1+s)^{-\alpha}\big)$ with $\alpha\ge 0 $, $K \gg 1$.
\end{itemize}
In the case $\phi(s)=1+\frac{1}{\sqrt{1+s}}$, the operator
$\Div \{ \phi( |\n u|^2 ) \n u \}$ is exactly the sum of 
the Laplacian $\Delta$ and the mean curvature operator
$\Div \left( \frac{\n u}{\sqrt{1+|\n u|^2}} \right)$. 
General quasilinear elliptic problems of the form:
\[
- \Div \{ \phi( |\n u|^2 ) \n u \} = g(u) \]
have been studied, for example, in \cite{ADP, CGD, DCSG,FLS,FIN,KS}.
Especially the assumption \eqref{phi2} is related with 
so-called $\Delta_2$-condition in the literature.

In our problem \ef{eq:1}, the quasilinear term depends not on $| \n u|^2$
but on $\frac{u^2+| \n u|^2}{2}$.
This particular quasilinear term appears in the study of
a nonlinear optics model describing the propagation of self-trapped beam 
in a cylindrical optical fiber made from a self-focusing dielectric material.
(See \cite{StZ1, StZ2} for the derivation.)
In this model, \eqref{phi2} plays a fundamental role as well.
We also refer to \cite{JR, Stu} for similar problems in a bounded domain. 

On the nonlinearity $g$, we require that
\begin{enumerate} [label=(g\arabic{*}),ref=g\arabic{*}]
\item \label{g1} $g\in C(\R,\R)$ and $g$ is odd;

\item \label{g2} there exists $m\in (-\phi_0, +\infty)$ such that
$$ 
-\infty <\liminf_{s\to 0} \frac{g(s)}{s} 
\le \limsup_{s \to 0} \frac{g(s)}{s} =-m;
$$

\item \label{g3}
denoting by $2^*=\frac{2N}{N-2}$, it holds
\begin{equation*} 
\displaystyle -\infty \le \limsup_{s \to +\infty}
\frac{g(s)}{s^{2^*-1}} \le 0;
\end{equation*}

\item \label{g4} there exists $\zeta>0$ such that $G(\zeta) >\Phi (\frac{\zeta^2}2)$,
where $G(s)=\int_0^{s} g(\tau) \,d\tau $. 
\end{enumerate}

Whenever $m \in (-\phi_0,0)$, instead of \eqref{phi2}, we assume 
\begin{enumerate}[label=($\phi$\arabic{*}'),ref=$\phi$\arabic{*}']
\setcounter{enumi}{1}
\item \label{phi2'} 
the map $t \mapsto \Phi\left( \frac{t^2}{2} \right)+\frac{m}{2}t^2$ 
is strictly convex on $\R$.
\end{enumerate}

The conditions \ef{g1}-\ef{g4} can be seen as a variant of
Berestycki-Lions' condition \cite{BL}
for the semilinear scalar field equation:
\begin{equation} \label{eq:2}
-\Delta u =g(u) \quad \hbox{in} \ \RN.
\end{equation}
Indeed, under additional assumptions on $\phi(s)$, 
one can show that \ef{g3} and \ef{g4} are {\it almost} optimal
for the existence of non-trivial solutions of \ef{eq:1}.
(See Theorem \ref{thm:5.4} below.)
An important feature is that the mass constant $m$ in \ef{g2} can be {\it negative}
because of the presence of a {\em mass term} in $\quas$ by \eqref{phi1}.
Our assumption \ef{g2} means that a total mass is positive so that
our problem \ef{eq:1} is actually the {\it positive mass case}. 

To state our main results, we prepare some notations.
By a {\it weak} solution of \ef{eq:1}, we mean a solution which satisfies \ef{eq:1} 
in the distribution sense,
equivalently, a critical point of the associated functional 
$I: \H \to \R$ defined by 
\begin{equation*}
I(u):= \irn \Phi \left(\frac{u^2+|\n u|^2 }2\right) \,dx - \irn G(u) \,dx.
\end{equation*} 
Our first result is the following one. 

\begin{theorem}\label{main1}
Assume \eqref{g1}-\eqref{g4}, \eqref{phi1} and 
\eqref{phi2} when $m \in [0,+\infty)$, or \eqref{phi2'} when $m \in (-\phi_0,0)$.
Then there exists a non-negative non-trivial weak solution of \eqref{eq:1}.
\end{theorem}

Next we consider the existence of a {\it regular} solution of \ef{eq:1}, 
that is, a solution which belongs to the class $C^1(\RN)$
(indeed $C^{1,\sigma}(\RN)$ for some $\sigma \in (0,1)$).
For this purpose, 
we impose the following slightly stronger condition on $\phi(s)$: 
\begin{enumerate}[label=($\phi$\arabic{*}),ref=$\phi$\arabic{*}]
\setcounter{enumi}{2}
\item\label{phi3} $\phi \in C^1\big([0,+\infty)\big)$ and 
there exists $C>0$ such that $s|\phi'(s)| \le C$ for all $s\in [0,+\infty)$. 
Moreover
\[
\phi_0 \le \phi(s) +2s \phi'(s) \quad \hbox{for all} \ \ s \in [0,+\infty).
\]
\end{enumerate} 
We notice that \eqref{phi3} implies \eqref{phi2} and \eqref{phi2'} respectively.
One can see that \eqref{phi3} is fulfilled if
\begin{itemize}
\item $\phi(s)=K+(1+s)^{-\alpha}$, with $0 \le \alpha \le \frac{1}{2}$, $K>0$; 
\item $\phi(s)=K-(1+s)^{-\alpha}$, with $\alpha \ge 0$, $K>1$.
\end{itemize} 

Then we have the following result. 

\begin{theorem}\label{main2}
Suppose that \eqref{phi1}, \eqref{phi3} and \eqref{g1}-\eqref{g4} hold. 
Then there exists a positive regular solution of \eqref{eq:1}, 
namely, a solution which is of the class $C^{1,\sigma}(\RN)$ for some $\sigma \in (0,1)$.
\end{theorem}

Once we have the regularity of solutions in hand,
we are able to apply the Pohozaev identity (see Lemma \ref{lem:5.1} below)
to obtain the existence of a {\it radial} ground state solution,
namely, a solution of \ef{eq:1} having 
least energy among all non-trivial radial solutions.

\begin{theorem}\label{main3}
Suppose that \eqref{phi1}, \eqref{phi3} and \eqref{g1}-\eqref{g4} hold. 
Then there exists a radial ground state solution of \eqref{eq:1}, 
which is of the class $C^{1,\sigma}(\RN)$ for some $\sigma \in (0,1)$ and
positive on $\RN$.
\end{theorem}

Finally, we are interested in the existence of a ground state solution
without restricting ourselves to the space of radial functions.
For this purpose, we need the following additional assumption on $\phi(s)$.

\begin{enumerate}[label=($\phi$\arabic{*}),ref=$\phi$\arabic{*}]
\setcounter{enumi}{3}
\item\label{phi4} $\phi \in C^2\big([0,+\infty)\big)$ and 
there exists $C>0$ such that $s^2|\phi''(s)| \le C$ for all $s\in [0,+\infty)$.
Moreover
\[
3\phi'(s)+2s|\phi''(s)| \le 0 \quad \hbox{and} \quad
0 \le \phi(s)+5s \phi'(s)-2s^2  |\phi''(s)|
 \quad \hbox{for all} \ s\in [0,+\infty). \]
\end{enumerate}
We observe that \eqref{phi4} requires 
\begin{equation}\label{phi'neg}
\phi'(s) \le 0 \quad \hbox{for all} \ s\in [0,+\infty). 
\end{equation}
This further implies that 
\begin{equation}\label{stima}
t^2\phi(t^2)\le  \Phi(t^2) \quad \hbox{for all} \ t\in \R.
\end{equation}
Elementary (but complicated) calculations show that 
\eqref{phi4} is satisfied for $\phi(s)=K+(1+s)^{-\alpha}$, with
$0 \le \alpha \le \frac{\sqrt{57}-7}{4}$, $K>0$ or
$\frac{\sqrt{57}-7}{4} < \alpha \le \frac{1}{2}$, $K \gg 1$. 
Under the stronger assumption \eqref{phi4}, 
we are able to obtain the following result,
which can be seen as an extension of the result in \cite{BL} for the semilinear case
$\phi(s) \equiv 1$. 

\begin{theorem} \label{main4}
Assume \eqref{phi1}, \eqref{phi3}, \eqref{phi4} and \ef{g1}-\ef{g4}.
Then \ef{eq:1} possesses a ground state solution $u$, 
namely, $u$ has least energy among all non-trivial solutions.
Moreover, $u$ is with fixed sign on $\RN$, 
radially symmetric with respect to some point
and of the class $C^{1,\sigma}(\RN)$ for some $\sigma \in (0,1)$. 
\end{theorem}

Here, we briefly introduce some ideas to obtain our main results.
By \eqref{phi1} and \ef{g1}-\ef{g4}, 
one can see that the functional $I$ has the mountain pass geometry. 
Then the existence of a non-trivial critical point of $I$ can be
shown by establishing the Palais-Smale condition.
Indeed once we could have the {\it boundedness of Palais-Smale sequences} in hand, 
one can expect the strong convergence of Palais-Smale sequences by 
decomposing the nonlinear term $g$ into two parts as in 
\cite{AP, BL, PW} and restricting ourselves to 
the space of radial functions.
However, as is well-known, the most difficult part is to prove the
boundedness of Palais-Smale sequences.

A standard strategy of constructing a bounded Palais-Smale sequence
is to apply so-called {\it monotonicity trick} as in \cite{J,S}.
However in the process of obtaining the boundedness, 
one needs to use the Pohozaev identity, 
causing the necessity of additional assumption on $\phi$ in our problem.
To avoid this technical difficulty, 
we adapt another strategy, namely, 
a technique of {\it adding one dimension of space} as established in \cite{HIT,jj}.
This approach enables us to construct a bounded Palais-Smale sequence
{\it without} using the Pohozaev identity.

Even if we could obtain the existence of a bounded Palais-Smale sequence,
we face another difficulty in our quasilinear problem.
In general, the boundedness of a Palais-Smale sequence $\{ u_n\}$ 
guarantees the weak convergence to some $u$ in $\H$ and $I'(u)=0$ from $I'(u_n) \to 0$.
Then the strong convergence can be obtained by considering the difference between 
$I'(u_n)[u_n] $ and $I'(u)[u]$.
However in our problem, one cannot show $I'(u)=0$ {\it a priori}, 
because the weak convergence of $u_n \rightharpoonup u$ in $\H$ provides us
{\it no information} of the pointwise convergence of $\n u_n(x) \to \n u(x)$.
To overcome this difficulty, in the case $m\ge 0$, 
we apply the theory of monotone operator \cite{DM}
for the functional 
\[
\Psi_0(u) := \irn \Quas \,dx, \quad u \in H^1(\RN) \]
to obtain the pointwise convergence of the gradient.
Then by considering $\Psi_0(u)-\Psi_0(u_n)-\Psi_0'(u_n)[u_n-u]$, 
together with the convexity of $\Phi(\frac{t^2}{2})$, 
we are able to prove the strong convergence of $u_n \to u$ in $\H$
and $I'(u)=0$ {\it a posteriori}. 
The case $m \in (-\phi_0,0)$ can be treated analogously. 

Finally in order to prove Theorem \ref{main4}, as in \cite{JT}, one establishes that
\begin{equation} \label{eq:3}
m_0 =b ,
\end{equation}
where
\begin{align*}
m_0 &:= \inf_{u \in S} I(u), \quad
S : = \{ u \in \H \setminus \{ 0\} \ ; \ I'(u)=0 \}, \\
b &:= \min_{u \in \P} I(u), \quad 
\P: = \{ u \in \H \setminus \{ 0\} \ ; \ u
\ \hbox{satisfies the Pohozaev identity for \ef{eq:1}} \}.
\end{align*}
In the semilinear case \ef{eq:2}, we can prove \ef{eq:3} by adapting
the scaling $\theta \mapsto u(\cdot/\theta)$
and using the variational characterization:
\begin{equation} \label{eq:4}
m_0 = \inf \left\{ \irn |\n u|^2 \,dx \ ; \ \irn G(u) \,dx =1 \right\}. 
\end{equation}
In our problem \ef{eq:1}, we cannot readily see that \ef{eq:3} holds
because of the loss of scaling property.
Once we could have the variational characterization \ef{eq:3} in hand, 
one can argue as in \cite{AW,CJS}.
Indeed \eqref{phi4} guarantees the convexity of some functions 
related with the Pohozaev identity for \ef{eq:1}.
This enables us to apply the generalized 
Polya-Szeg\"o inequality \cite{HK} to prove that minimizing sequences for $m_0$ 
can be assumed to be radially symmetric.
Then the existence of a ground state of \ef{eq:1} can be shown similarly as Theorem \ref{main1}.

\smallskip
This paper is organized as follows.
In Section \ref{se2}, we investigate some basic properties of the operator $\Psi_0$.
We prove Theorem \ref{main1} by applying the mountain pass theorem in Section \ref{se3}.
Section \ref{se4} is devoted to the study of regularity and positivity of
weak solutions of \ef{eq:1}.
Theorems \ref{main2} and \ref{main3} will be shown in Section \ref{se5}.
Finally, we prove Theorem \ref{main4} in Section \ref{se6}.

We conclude this introduction fixing some notations. 
For any $p\ge 1$, we denote by $L^p(\RN)$ the usual Lebesgue spaces 
equipped by the standard norm $\|\cdot\|_{L^p}$. 
In our estimates, we will frequently denote by $C>0$ fixed constants, 
that may change from line to line, 
but are always independent of the variable under consideration. 
We also use the notations $o_n(1)$ to describe a quantity which goes to zero as $n \to +\infty$.
Moreover, for any $R>0$, we denote by $B_R$ the ball of $\RN$ 
centered in the origin with radius $R$.

\section{Fundamental properties of the operator $\Psi_0$}\label{se2}

In this section, we investigate fundamental properties of the operator
$\Psi_0: H^1(\RN) \to \R$ defined by
\[
\Psi_0(u) := \irn \Quas \,dx, \quad u \in H^1(\RN).\]
For later use, we also denote,
\begin{align*}
\Psi_m(u) &:= \Psi_0(u) +\frac{m}{2} \irn u^2 \,dx, \\
\tilde{\Psi}_m(u) &:= \irn \left\{ \Quas + m \left( \frac{u^2+|\n u|^2}{2} \right) \right\} \,dx
= \Psi_m(u) + \frac{m}{2} \irn |\n u|^2 \,dx,
\end{align*}
where $m\in (-\phi_0,+\infty)$ is defined in \eqref{g2}.
Then one readily finds that
\begin{align}
\Psi_0(u) \le \Psi_m(u) \le \tilde{\Psi}_m(u) &\quad \hbox{if} \ m\in [ 0,+\infty), \label{eq:2.1} \\
\tilde{\Psi}_m(u) \le \Psi_m(u) \le \Psi_0(u) &\quad \hbox{if} \ m\in (-\phi_0,0). \label{eq:2.2} 
\end{align}
Moreover from \eqref{phi1}, it is standard to prove that
$\Psi_0$, $\Psi_m$, $\tilde{\Psi}_m \in C^1(\H,\R)$ and
\[
\Psi_m'(u)[\varphi] = \irn \left\{\quas ( u \varphi +\nabla u \cdot \n \varphi) +m u \varphi \right\}\,dx,
\quad \varphi \in H^1(\RN).\]

First we recall the following convergence result, 
which appears in the theory of monotone operators
\cite[Lemma 6]{DM}.

\begin{proposition} \label{prop:2.1}
Let $X$ be a finite dimensional real Hilbert space with norm $| \cdot |$ 
and inner product $\langle \cdot,\cdot \rangle$.
Suppose that $\beta :X \to X$ is a continuous function which is 
strictly monotone, i.e.
\[
\langle \beta(\xi)-\beta(\bar{\xi}), \xi-\bar{\xi} \rangle >0
\quad \hbox{for every} \ \xi,\bar{\xi} \in X \ \hbox{with} \ \xi \ne \bar{\xi}.\]
Let $\{ \xi_n \} \subset X$ and $\xi \in X$ be such that
\[
\lim_{n \to +\infty} \langle \beta(\xi_n)-\beta(\xi), \xi_n-\xi \rangle =0.\]
Then $\{ \xi_n \}$ converges to $\xi$ in $X$.
\end{proposition}

Using Proposition \ref{prop:2.1}, we are able to obtain the following
pointwise convergence result of the gradient.

\begin{lemma} \label{lem:2.2}
Assume \eqref{phi1} and \eqref{phi2} when $m\in [0,+\infty)$, or \eqref{phi2'} when $m\in (-\phi_0,0)$. 
If $\{ u_n \} \subset \H$ satisfies $u_n \rightharpoonup u$ weakly in $\H$
and $\Psi_m'(u_n)[u_n-u] \to 0$ as $n \to +\infty$, then
\[
\nabla u_n(x) \to \n u(x) \quad{as} \ n\to +\infty \ \hbox{a.e. in} \ \R^N.
\]
\end{lemma}

\begin{proof}
First since $u_n \rightharpoonup u$ weakly in $\H$, 
it follows that $\Psi_m'(u)[u_n-u] \to 0$ as $n \to +\infty$.
Hence by the assumption, we get
\begin{align} \label{eq:2.3}
o_n(1) &= \big( \Psi_m'(u_n)-\Psi_m'(u) \big) [u_n-u] \\
&= \big( \Psi_0'(u_n) -\Psi_0'(u) \big) [u_n-u] +m \irn (u_n-u)^2\,dx. \notag
\end{align}
We distinguish the proof into the cases $m\in [0,+\infty)$ and $m\in (-\phi_0,0)$.

When $m\in [0,+\infty)$, \ef{eq:2.3} yields that
\begin{equation} \label{eq:2.4}
\limsup_{n \to +\infty} \big( \Psi_0'(u_n)-\Psi_0'(u) \big) [u_n-u] \le 0.
\end{equation}
On the other hand, a direct computation shows that
\begin{align*}
&\big( \Psi_0'(u_n)-\Psi_0'(u) \big) [u_n-u] \\
&= \irn \left\{ \quasn \n u_n - \quas \n u \right\} \cdot \nabla (u_n-u) \,dx \\
&\quad + \irn \left\{ \quasn u_n - \quas u \right\} (u_n-u) \,dx.
\end{align*}
Putting $X= \R \times \R^N$, $\xi_n= (u_n(x), \nabla u_n(x))$,
$\xi=(u(x),\nabla u(x))$ and 
$\beta(\xi)= \frac{\xi}{\sqrt{2}} \phi \left( \frac{|\xi|^2}{2} \right)$, 
one finds that
\begin{equation} \label{eq:2.5}
\big( \Psi_0'(u_n)-\Psi_0'(u) \big) [u_n-u]
= \sqrt{2} \irn \langle \beta(\xi_n)-\beta(\xi),\xi_n-\xi \rangle \,dx, 
\end{equation}
where $\langle \cdot,\cdot \rangle$ is the standard inner product on $\R^{N+1}$.
Moreover by \eqref{phi2}, as performed in \cite{JR} and \cite{Stu}, it follows that
\begin{align} 
\langle \beta(\xi)-\beta(\bar{\xi}), \xi-\bar{\xi} \rangle 
&= \frac{1}{\sqrt{2}} \phi \left( \frac{|\xi|^2}{2} \right) |\xi -\bar{\xi}|^2>0
\quad \hbox{whenever} \ |\xi|=|\bar{\xi}| \ \hbox{and} \ \xi \ne \bar{\xi}, \notag \\
\langle \beta(\xi)-\beta(\bar{\xi}), \xi-\bar{\xi} \rangle 
&=\phi \left( \frac{|\xi|^2}{2} \right) 
\left\langle \frac{\xi}{\sqrt{2}}, \xi-\bar{\xi} \right\rangle
-\phi \left( \frac{|\bar{\xi}|^2}{2} \right) 
\left\langle \frac{\bar{\xi}}{\sqrt{2}}, \xi-\bar{\xi} \right\rangle \label{eq:260} \\
&\ge \left\{ \frac{|\xi|}{\sqrt{2}} \phi \left( \frac{|\xi|^2}{2} \right) 
- \frac{|\bar{\xi}|}{\sqrt{2}} \phi \left( \frac{|\bar{\xi}|^2}{2} \right) \right\}
(|\xi|- |\bar{\xi}|) 
>0 \quad \hbox{whenever} \ |\xi| \ne |\bar{\xi}|. \label{eq:2.6} 
\end{align} 
Thus from \ef{eq:2.4} and \ef{eq:2.5}, we obtain
\[
0= \lim_{n \to +\infty} \irn \langle \beta(\xi_n)-\beta(\xi),\xi_n-\xi \rangle \,dx\]
and hence
\begin{equation} \label{eq:2.7}
\langle \beta(\xi_n)-\beta(\xi),\xi_n-\xi \rangle \to 0 
\quad \hbox{as} \ n\to +\infty \ \hbox{a.e. in} \ \R^N.
\end{equation}
By \eqref{eq:260}, \ef{eq:2.6} and \ef{eq:2.7}, we are able to apply Proposition \ref{prop:2.1}
to conclude that $\xi_n \to \xi$ a.e. in $\R^{N+1}$ and so
$\n u_n(x) \to \n u(x)$ a.e. in $\R^N$.

Next we consider the case $m\in (-\phi_0,0)$.
In this case, we first find that
\begin{align*}
o_n(1) &= \big( \Psi_m'(u_n)-\Psi_m'(u) \big)[u_n-u] \\
&= \big( \tilde{\Psi}_m'(u_n)-\tilde{\Psi}_m'(u) \big)[u_n-u]
-m \irn | \nabla (u_n-u) |^2\,dx.
\end{align*}
Putting $\beta_m(\xi)= \left( \phi \left( \frac{|\xi|^2}{2} \right)+m \right) 
\frac{\xi}{\sqrt{2}}$, one gets from $m<0$ that
\begin{equation} \label{eq:2.8}
\limsup_{n \to +\infty} \irn \langle \beta_m(\xi_n)-\beta_m(\xi), \xi_n-\xi \rangle \,dx \le 0.
\end{equation}
Moreover since $m>-\phi_0$, we have from \eqref{phi2'} that
the function $t \mapsto \left( \phi \left( \frac{t^2}{2} \right) +m \right) \frac{t}{\sqrt{2}}$ 
is increasing on $[0,+\infty)$. 
Arguing as done previously, this implies that 
\begin{equation} \label{eq:2.9}
\langle \beta_m(\xi)-\beta_m(\bar{\xi}),\xi-\bar{\xi} \rangle >0
\quad \hbox{for any} \ \xi, \bar{\xi} \ \hbox{with} \ \xi \ne \bar{\xi}.
\end{equation}
Then from \ef{eq:2.8} and \ef{eq:2.9}, we conclude as above.
\end{proof}

Finally we establish the following Brezis-Lieb type convergence result.

\begin{lemma} \label{lem:2.3}
Assume \eqref{phi1} and \eqref{phi2} when $m\in [0,+\infty)$, or \eqref{phi2'} when $m\in (-\phi_0,0)$. 
If $\{ u_n \} \subset \H$ satisfies $u_n \rightharpoonup u$ weakly in $\H$,
$\nabla u_n(x) \to \nabla u(x)$ a.e. in $\R^N$ and
\[
\limsup_{n \to +\infty} \Psi_m(u_n) \le \Psi_m(u). \]
Then, up to a subsequence, $u_n \to u$ strongly in $\H$.
\end{lemma}

\begin{proof}
Again, we distinguish the cases $m \in [0,+\infty)$ and $m \in (-\phi_0,0)$.
First we suppose that $m \in [0,+\infty)$. 
One finds that 
\begin{equation} \label{eq:2.10}
\Psi_m(u_n)-\Psi_m(u)-\Psi_m(u_n-u) 
= \Psi_0(u_n)-\Psi_0(u)-\Psi_0(u_n-u)
+m \irn u (u_n-u) \,dx. 
\end{equation}
Now setting 
\[
f_n(x) := \left( \frac{u_n(x)^2 +|\n u_n(x)|^2}{2} \right)^{\frac{1}{2}}, \ 
f(x) := \left( \frac{u(x)^2 +|\n u(x)|^2}{2} \right)^{\frac{1}{2}}, \]
we may assume that $f_n(x) \to f(x)$ a.e. in $\R^N$.
Moreover putting $j(s)=\Phi(s^2)$, one knows from \eqref{phi2} that
$j(s)$ is continuous, convex on $\R$ and $j(0)=0$.
Then by the Brezis-Lieb lemma \cite[Theorem 2 and Examples (b)]{BrLi},
it follows that
\[
\lim_{n \to +\infty} \irn | j(f_n)-j(f)-j(f_n-f)| \,dx =0 \]
and hence
\begin{equation} \label{eq:2.11}
\lim_{n \to +\infty} \big\{ \Psi_0(u_n)-\Psi_0(u)-\Psi_0(u_n-u) \big\}=0.
\end{equation}
From \ef{eq:2.10} and \ef{eq:2.11}, $\Psi_m \ge 0$ and by the assumptions of this lemma,
we get
\[
0 \le \liminf_{n \to +\infty} \Psi_m(u_n-u) \le \limsup_{n \to +\infty} \Psi_m(u_n-u)
= \limsup_{n \to +\infty} \big\{ \Psi_m(u_n)-\Psi_m(u) +o_n(1) \big\} \le 0, 
\]
from which we conclude
\[
\lim_{n \to +\infty} \Psi_m(u_n-u)=0. \]
Since, from \ef{eq:2.1}, $0 \le \Psi_0(u_n-u) \le \Psi_m(u_n-u)$ when $m \in [0,+\infty)$,
this also yields that
\[
\lim_{n \to +\infty} \Psi_0(u_n-u) =0. \]
Then from \eqref{phi1}, we conclude that
$u_n - u \to 0$ in $\H$, as desired.

Next we assume that $m \in (-\phi_0,0)$.
First one observes that
\[
\Psi_m(u_n)-\Psi_m(u)-\Psi_m(u_n-u)
= \tilde{\Psi}_m(u_n)-\tilde{\Psi}_m(u)-\tilde{\Psi}_m(u_n-u)
-m \irn \nabla u \cdot \nabla (u_n-u) \,dx. \]
Letting $j_m(s)=\Phi(s^2)+m s^2$, we have from \eqref{phi2'} that
$s \mapsto j_m(s)$ is convex.
Then we can apply the Brezis-Lieb lemma to obtain
\[
\lim_{n \to +\infty} \big\{ \tilde{\Psi}_m(u_n)-\tilde{\Psi}_m(u)-\tilde{\Psi}_m(u_n-u) 
\big\} =0. \]
In a similar argument as the case $m \in [0,+\infty)$, it follows that
\[
\lim_{n \to +\infty} \Psi_m(u_n-u) =0. \]
Since, from \ef{eq:2.2} and \eqref{phi1}, 
\[
\Psi_m(u) \ge \tilde{\Psi}_m(u) \ge \frac{\phi_0+m}{2} \irn \left(|\nabla u|^2 +u^2 \right)\,dx 
\quad \hbox{if} \ m \in (-\phi_0,0), \]
we deduce that $u_n -u \to 0$ in $\H$.
\end{proof}

\section{Variational setting and existence of a non-negative non-trivial solution}\label{se3}

In this section, 
we perform a variational setting of \ef{eq:1} and prove the existence of a 
non-negative non-trivial solution of \ef{eq:1}.

First as in \cite{AP, BL, PW}, 
we decompose the nonlinear term $g$ as follows.
Let $s_0:= \min \{ s\in [\zeta, +\infty) \ ; \ g(s)=0 \}$, 
$s_0=+\infty$ if $g(s) \ne 0$ for any $s \ge \zeta$ and
$\tilde{g}:\R \to \R$ be the continuous function such that
\begin{equation} \label{eq:3.1}
\tilde g(s)=\left\{
\begin{array}{ll}
g(s) &\hbox{ on } [0,s_0],
\\
0 &\hbox{ on } \R_+\setminus [0,s_0]
\end{array}
\right. \quad \hbox{for} \ s \ge 0.
\end{equation}
For $s <0$, $\tilde{g}$ is defined by $\tilde{g}(s)=-g(-s)$.
As we will see in Proposition \ref{prop:3.1}, any weak solution $u \in \H$
of \ef{eq:1} with $\tilde{g}$ in the place of $g$ satisfies $|u| \le s_0$.
Thus hereafter we may replace $g$ by $\tilde{g}$, so that
$g$ fulfills \ef{g1}, \ef{g2}, \ef{g4} and a stronger condition
\begin{equation} \label{eq:3.2}
\lim_{s\to\pm \infty} \frac{|g(s)|}{|s|^{2^*-1}}=0.
\end{equation}

Next we put
\begin{equation*}
g_1(s):=
\begin{cases}
\left(g(s)+ms \right)^+ & \hbox{for }s\ge 0,
\\
0 & \hbox{for }s< 0,
\end{cases}
\end{equation*}
$g_2(s):=g_1(s)-g(s)$ for $s\ge 0$ and 
extend $g_2(s)$ as an odd function for $s<0$.
Then from \ef{g2} and \ef{eq:3.2}, one has
\begin{align}
\lim_{s\to 0} \frac{g_1(s)}{s} &= 0,\label{eq:3.3} \\
\lim_{s\to\pm\infty} \frac{g_1(s)}{|s|^{2^*-1}}&=0 \label{eq:3.4}
\end{align}
and
\begin{equation} \label{eq:3.5}
g_2(s) \ge ms \quad \hbox{for} \ s \ge 0.
\end{equation}
We set $G_i(s)=\int_0^s g_i(\tau) \,d\tau$, $i=1,2$.
Observing that $G_1(s)=0$ for $s<0$, 
we also deduce that, for any $\eps>0$, there exists $C_{\eps}>0$ such that
\begin{equation} \label{eq:3.6}
0 \le g_1(s) \le \eps |s| + C_\eps |s|^{2^*-1} \quad \hbox{for} \ s\in \R,
\end{equation}
\begin{equation} \label{eq:3.7}
0 \le G_1(s) \le \frac \eps 2s^2 + \frac {C_\eps} {2^*} |s|^{2^*} \quad \hbox{for} \ s\in \R.
\end{equation}
Moreover from \ef{g2}, \ef{eq:3.2} and \ef{eq:3.7}, 
for suitable $c_1$, $c_2>0$, it holds that
\begin{equation} \label{eq:3.8}
\frac{ms^2}{2} \le G_2(s)=G_2(|s|)
\le G_1(|s|)+|G(|s|)|
\le c_1|s|^2+c_2|s|^{2^*} \quad \hbox{for} \ s\in \R.
\end{equation}

Under these preparations, we define the functional
\begin{equation*}
I(u):= \irn \Phi \left(\frac{u^2+|\n u|^2 }2\right) \,dx 
+\irn G_2(u) \,dx - \irn G_1(u) \,dx.
\end{equation*}
Then from \eqref{phi1}, \ef{eq:3.7} and \ef{eq:3.8}, $I$ is well-defined 
and $C^1$ on $\H$.

\begin{proposition} \label{prop:3.1}
Assume \eqref{g1}-\eqref{g3} and \eqref{phi1}.
If $u \in \H$ is a critical point of $I$, then $u$ is a weak solution of \ef{eq:1}.
\end{proposition}

\begin{proof}
It suffices to show that $0 \le u \le s_0$.
Letting $u^-:=\min \{0,u\}$, we have $u^-\in \H$ and hence $I'(u)[u^-]=0$. 
We denote $\O:=\{x\in \RN \ ; \ u(x)<0\}$. 
Then from \eqref{phi1}, \ef{eq:3.5} and the fact $g_1(s)=0$ for $s<0$ from \ef{eq:3.1},
one finds that
\begin{align*}
\phi_0 \|\n u\|^2_{L^2(\O)} + (\phi_0+m) \|u\|^2_{L^2(\O)}
&\le \int_\O \phi \left(\frac{u^2+|\n u|^2 }2\right) \left(u^2+|\n u|^2\right) \,dx
+\int_\O g_2(u)u \,dx \\
& = \int_\O g_1(u)u \,dx =0,
\end{align*}
by which, since $\phi_0+m>0$, we deduce that $u^-=0$ on $\RN$.

Next we put $\O_0:= \{ x\in \RN \ ; \ u(x) >s_0 \}$
and assume that $\O_0 \ne \emptyset$.
Since $g(u)=0$ on $\O_0$, it follows that 
\[
\quas u - g(u) \ge 0 \quad \hbox{on} \ \O_0.\]
This implies that $u \in \H$ is a weak solution of the differential inequality:
\[
\Div \bm{A}(u, \n u) \ge 0 \quad \hbox{in} \ \O_0,
\]
where $\bm{A}(u,\bp)=\quasi \bp$ for $(u,\bp) \in \R \times \RN$.
From \eqref{phi1}, it follows that
\[
\bm{A}(u,\bp) \cdot \bp = \quasi | \bp|^2 \ge \phi_0 |\bp|^2, \]
and hence we can apply the weak maximum principle \cite[Theorem 3.2.1]{PuS}
to conclude that
\[
\max_{\O_0} u(x) = \max_{\partial \O_0} u(x).\]
This is a contradiction to the definition of $\O_0$.
Thus it holds that $\O_0 = \emptyset$ and $u(x) \le s_0$ a.e. in $\RN$.
This completes the proof.
\end{proof}

Hereafter, we work on the function space:
\begin{equation*}
H^1_r(\R^N) =\{ u\in\H \ ;\ u \hbox{ is radial} \}
\end{equation*}
and set $\| u\|^2:= \irn (|\n u|^2 +u^2) \,dx$.
Next we establish the mountain pass geometry.

\begin{lemma} \label{lem:3.2}
Assume \eqref{g1}-\eqref{g4} and  \eqref{phi1}.
Then the functional $I: H_r^1(\RN) \to \R$ has the mountain pass geometry, i.e.
\begin{itemize}
\item[\rm(i)] there exist $\alpha$, $\rho>0$ such that
$I(u) \ge \alpha$ for $\|u\|=\rho$;
\item[\rm(ii)] there exists $z \in \Hr $ with $\|z \|>\rho$ such that $I(z)<0$. 
\end{itemize}
\end{lemma}

\begin{proof}
(i). \ 
From \eqref{phi1}, \ef{eq:3.7} and \ef{eq:3.8}, 
taking $\eps \in (0, \phi_0+m)$, we have
\begin{align*}
I(u) & \ge \frac{\phi_0}{2} \irn( |\n u|^2 + u^2) \,dx 
+\irn G_2 (u) \,dx -\irn G_1 (u) \,dx \\
& \ge \frac{\phi_0}{2} \| \n u \|^2_{L^2} 
+ \frac{\phi_0+m-\eps}{2} \| u \|^2_{L^2} - \frac{C_\eps}{2^*} \irn |u|^{2^*} \,dx
\quad \hbox{for any} \ u \in H_r^1(\RN).
\end{align*}
Then, by the Sobolev inequality, 
there exist $\alpha$, $\rho>0$ such that $I(u) \ge \alpha$ for $\| u \|= \rho$.

(ii). \
Following \cite{BL}, for any $R>0$, we set $w_R\in \Hr$ so that
\[
w_R(x)=
\begin{cases}
\zeta & \hbox{ if }|x|\le R, \\
\zeta(R+1-|x|) & \hbox{ if }R\le |x|\le R+1, \\
0 & \hbox{ if }|x|\ge R+1,
\end{cases}
\]
where $\zeta$ is defined in \eqref{g4}.
Using the monotonicity of the map $t \mapsto \Phi(t^2)$ on $[0,+\infty)$, we have
\begin{align*}
I(w_R)&= \irn \Phi \left(\frac{w_R^2 + |\n w_R|^2}2\right) \,dx
+\irn G_2(w_R) \,dx -\irn G_1(w_R) \,dx \\
&=\int_{B_R} \Phi \left(\frac{\zeta^2}2\right) \,dx
+\int_{B_{R+1}\setminus B_R} \Phi \left(\frac{w_R^2 + |\n w_R|^2}2\right) \,dx \\
&\qquad+\int_{B_R} G_2(\zeta) \,dx
+\int_{B_{R+1}\setminus B_R} G_2(w_R) \,dx 
-\int_{B_R} G_1(\zeta) \,dx 
-\int_{B_{R+1}\setminus B_R} G_1(w_R) \,dx \\
&\le \int_{B_R} \left\{ \Phi \left(\frac{\zeta^2}2\right) +G_2(\zeta)- G_1(\zeta)\right\}
\,dx 
+\int_{B_{R+1}\setminus B_R} \left\{ \Phi \left(\zeta^2\right)
+G_2(w_R)-G_1(w_R) \right\} \,dx \\
&\le C \left\{ \Phi \left(\frac{\zeta^2}2\right) - G(\zeta)\right\} R^N 
+C \max_{s\in [0,\zeta] }\left| \Phi \left(\zeta^2\right)
- G(s)\right| R^{N-1}.
\end{align*}
By \eqref{g4}, we can find sufficiently large $R>0$ so that $I(w_R)<0$ and $\|w_R\|>\rho$.
Putting $z=w_R$, we finish the proof.
\end{proof}

By Lemma \ref{lem:3.2}, denoting
\begin{equation*}
\Gamma := \left\{ \g \in C\big([0,1],\Hr\big) \ ; \ \g(0)=0,
I(\g(1))< 0\right\},
\end{equation*}
we infer that $\G$ is non-empty and
\begin{equation*}
c:=\inf_{\g \in \G }\max_{t\in[0,1]}
I(\g(t)) \ge \alpha>0.
\end{equation*}

Now, following \cite{HIT,jj}, we define the  functional $J:\R\times \Hr\to \R$  as 
\begin{align*}
J(\theta, u)&= I\big(u(e^{-\theta}\cdot )\big) \\
&=e^{N\theta} \irn \Phi \left(\frac{u^2+e^{-2\theta}|\n u|^2}{2}\right) \,dx
+e^{N\theta}\irn G_2(u) \,dx -e^{N\theta}\irn G_1(u) \,dx.
\end{align*}
With similar arguments of Lemma \ref{lem:3.2}, 
$J$ also has the mountain pass geometry and we can define its mountain pass level as
\[
\tilde c:=\inf_{(\theta,\g) \in \Sigma\times \G} \max_{t\in [0,1]}J\big(\theta(t),\g(t)\big), 
\]
where 
\[
\Sigma:=\left\{\theta \in C\big([0,1],\R\big) \ ; \ \theta(0)=\theta(1)=0\right\}.
\]
Arguing as in \cite[Lemma 4.1]{HIT}, we derive the following.

\begin{lemma} \label{lem:3.3}
The mountain pass levels of $I$ and $J$ coincide, namely $c=\tilde c$.
\end{lemma}

Now, as a immediate consequence of Ekeland's variational principle, 
we have the result below, whose proof follows as in \cite[Lemma 2.3]{jj}.

\begin{lemma} \label{lem:3.4}
Let $\eps>0$. Suppose that $\eta \in \Sigma \times \G$ satisfies 
\[
\max_{t \in [0,1]}J( \eta(t))\le \tilde c+\eps.
\]
Then there exists $(\theta, u)\in \R\times \Hr$ such that
\begin{enumerate}
\item[\rm(i)] ${\rm dist}_{\R \times \Hr}\big((\theta,u),\eta([0,1])\big)
\le 2 \sqrt{\eps}$;
\item[\rm(ii)] $J(\theta,u)\in [c-\eps, c+\eps]$;
\item[\rm(iii)] $\|D J(\theta,u)\|_{\R \times (\Hr)'}\le 2 \sqrt{\eps}$.
\end{enumerate}
\end{lemma}

Arguing as in \cite[Proposition 4.2]{HIT}, 
by Lemmas \ref{lem:3.3} and \ref{lem:3.4}, the following proposition holds

\begin{proposition} \label{prop:3.5}
There exists a sequence $\{(\theta_n,u_n)\} \subset \R \times \Hr$ such that, 
as $n \to +\infty$, 
\begin{enumerate}
\item[\rm(i)] $\theta_n \to 0$;
\item[\rm(ii)] $J(\theta_n,u_n)\to c$;
\item[\rm(iii)] $\de_\theta J(\theta_n,u_n)\to 0$;
\item[\rm(iv)] $\de_u J(\theta_n,u_n)\to 0$ strongly in $(\Hr)'$. 
\end{enumerate}
\end{proposition}

Next we recall the Strauss compactness lemma. 
(See \cite[Theorem A.1]{BL}, \cite{Str}.) 
It will be a fundamental tool in our arguments. 

\begin{lemma} \label{lem:3.6}
Let $P$ and $Q:\R\to\R$ be two continuous functions satisfying
\begin{equation*} 
\lim_{|s|\to+\infty}\frac{P(s)}{Q(s)}=0.
\end{equation*}
Suppose that $\{v_n\}$ and $v$ are measurable functions from $\RN$ to $\R$
such that
\begin{align*}
&\sup_{n \in \mathbb{N}} \int_\RN | Q(v_n(x))|\,dx <+\infty \quad \hbox{and} \quad 
P(v_n(x))\to v(x) \:\hbox{a.e. in }\RN. 
\end{align*}
Then $\|(P(v_n)-v)\|_{L^1(B)}\to 0$ for any bounded Borel set $B$.

Moreover, if we have also
\begin{align*}
\lim_{s\to 0}\frac{P(s)}{Q(s)} &=0 \quad \hbox{and} \quad
\lim_{x\to\infty}\sup_{n \in \mathbb{N}} |v_n(x)| = 0, 
\end{align*}
then $\|(P(v_n)-v)\|_{L^1(\RN)}\to 0.$
\end{lemma}

Now we are ready to prove Theorem \ref{main1}.

\begin{proof}[Proof of Theorem \ref{main1}]
First by Proposition \ref{prop:3.5}, 
there exists a sequence $\{(\theta_n,u_n)\} \subset \R \times \Hr$ such that
\begin{align} \label{eq:3.9}
&e^{N\theta_n} \irn \Phi \left(\frac{u_n^2+e^{-2\theta_n}|\n u_n|^2}{2}\right) \,dx
+e^{N\theta_n} \irn G_2(u_n) \,dx
-e^{N\theta_n}\irn G_1(u_n) \,dx \\
&=c+o_n(1), \notag
\end{align}
\begin{align} \label{eq:3.10}
& Ne^{N\theta_n}\irn \Phi \left(\frac{u_n^2+e^{-2\theta_n}|\n u_n|^2}{2}\right) \,dx
-e^{(N-2)\theta_n}\irn \phi 
\left(\frac{u_n^2+e^{-2\theta_n}|\n u_n|^2}{2}\right)|\n u_n|^2 \,dx \\
&+Ne^{N\theta_n}\irn G_2(u_n) \,dx -Ne^{N\theta_n}\irn G_1(u_n) \,dx =o_n(1), \notag
\end{align}
and, for all $\varphi\in \Hr$,
\begin{align} \label{eq:3.11}
&e^{N\theta_n}\irn \phi \left(\frac{u_n^2+e^{-2\theta_n}|\n u_n|^2}{2}\right)
\big(u_n \varphi +e^{-2\theta_n} \n u_n \cdot \n \varphi \big) \,dx \\
&+e^{N\theta_n}\irn g_2(u_n) \varphi \,dx -e^{N\theta_n} \irn g_1(u_n) \varphi\,dx 
=o_n(1)\| \varphi \|. 
\notag
\end{align}
By \ef{eq:3.9} and \ef{eq:3.10}, we have 
\begin{equation} \label{eq:3.12}
\frac{e^{(N-2)\theta_n}}N\irn \phi 
\left(\frac{u_n^2+e^{-2\theta_n}|\n u_n|^2}{2}\right)|\n u_n|^2 \,dx = c+o_n(1).
\end{equation}
From \eqref{phi1} and since $\theta_n \to 0$ as $n\to +\infty$, 
$\{u_n\}$ is bounded $D^{1,2}(\RN)$, namely 
\begin{equation*} 
\|\n u_n \|_{L^2} \le C \quad \hbox{ for some }C>0 \hbox{ and for all }n\ge 1. 
\end{equation*}
Moreover using \eqref{eq:3.7}, \ef{eq:3.8}, \ef{eq:3.10} and \eqref{phi1}, 
one also has
\begin{align*}
\frac{\phi_0+m }{2}\|u_n\|_{L^2}^2
&\le \irn \Phi \left(\frac{u_n^2+e^{-2\theta_n}|\n u_n|^2}{2}\right) \,dx 
+\irn G_2(u_n) \,dx \\
&=\frac{e^{-2\theta_n}}N\irn \phi \left(\frac{u_n^2+e^{-2\theta_n}|\n u_n|^2}{2}\right)|\n u_n|^2 \,dx
+\irn G_1(u_n) \,dx +o_n(1) \\
& \le \frac{2\phi_1}{N}\|\n u_n\|^2_{L^2}
+\frac \e 2\|u_n\|^2_{L^2}
+\frac{C_\e}{2^*}\|u_n\|_{L^{2^*}}^{2^*} +o_n(1).
\end{align*}
Choosing $\e>0$ sufficiently small so that $\phi_0+m-\e>0$ 
and taking in account the embedding of $D^{1,2}(\RN)$ into $L^{2^*}(\RN)$,
we find that $\{u_n\}$ is bounded also in $L^2(\RN)$ 
and hence 
\begin{equation} \label{eq:3.13}
\| u_n \| \le C \quad \hbox{ for some }C>0 \hbox{ and for all }n\ge 1. 
\end{equation}
This implies the existence of $u\in \Hr$ such that 
\begin{equation} \label{eq:3.14}
u_n\rightharpoonup u\;\hbox{weakly in }\H
\end{equation}
and
\begin{equation} \label{eq:3.15}
u_n(x)\to u(x)\;\hbox{a.e. in }\RN.
\end{equation}
Moreover by the radial lemma (see \cite{BL,Str}), 
it follows that 
\begin{equation} \label{eq:3.16}
|u_n(x)| \le C |x|^{-\frac{N-2}{2}} \| \n u_n \|_{L^2}, \ |x| \ge 1
\quad \hbox{for some} \ C>0 \ \hbox{independent of} \ n\in \N.
\end{equation}

Now from \ef{eq:3.3}, \ef{eq:3.4}, \ef{eq:3.13}, \ef{eq:3.15} and \ef{eq:3.16}, 
we are able to apply Lemma \ref{lem:3.6} 
provided that $P(s)=g_1(s)s$, $Q(s)= s^2+|s|^{2^*}$,
$\{v_n\}=\{u_n\}$, and $v=g_1(u)u$.
Then we deduce that
\begin{equation*}
\irn g_1(u_n)u_n \,dx \to \irn g_1(u)u \,dx \quad \hbox{as} \ n \to +\infty.
\end{equation*}
Moreover from \ef{eq:3.6},
one finds that $g_1(u) \in (\Hr)'$ and hence
\[
\irn g_1(u)(u_n-u) \,dx \to 0 \quad \hbox{as} \ n\to +\infty. \]
Thus we obtain
\begin{align} \label{eq:3.17}
& \irn g_1(u_n)(u_n-u) \,dx \\
&= \irn g_1(u_n)u_n \,dx - \irn g_1(u)u \,dx 
+\irn g_1(u)(u_n-u) \,dx 
\to 0 \quad \hbox{as} \ n\to +\infty. \notag
\end{align}
Next we put 
\[ 
h(s):=g_2(s)-ms \quad \hbox{ for } s\ge 0
\]
and extend it as an odd function for $s<0$.
By \ef{eq:3.5}, we can observe that $h(s)/s\ge 0$ for any $s \neq 0$. 
Thus from \ef{g2} and \ef{eq:3.3}, one has
\begin{equation*}
0 \le \liminf_{s\to 0}\frac{h(s)}{s}
\le \limsup_{s\to 0}\frac{h(s)}{s}
= \limsup_{s\to 0}\frac{g_2(s)-ms}{s}
= \limsup_{s\to 0}\frac{g_1(s)-g(s)-ms}{s}=0
\end{equation*}
and hence
\begin{equation} \label{eq:3.18}
\lim_{s \to 0} \frac{h(s)}{s}=0. 
\end{equation}
Furthermore from \ef{eq:3.2} and \ef{eq:3.4}, we also have
\begin{align} 
\lim_{s \to +\infty} \frac{h(s)}{|s|^{2^*-1}}
&=\lim_{s \to + \infty} \frac{g_1(s)-g(s)-ms}{|s|^{2^*-1}}=0, \label{eq:3.19}
\\
\lim_{s \to -\infty} \frac{h(s)}{|s|^{2^*-1}}
&=\lim_{s \to - \infty} \frac{-g_1(-s)-g(s)-ms}{|s|^{2^*-1}}=0. \label{eq:3.19-2}
\end{align}
By \ef{eq:3.18}, \eqref{eq:3.19} and \ef{eq:3.19-2}, 
we can apply once again Lemma \ref{lem:3.6} to conclude that
\[
\irn h(u_n)u_n \,dx \to \irn h(u)u \,dx \quad \hbox{as} \ n \to +\infty. \]
Repeating previous arguments, it holds that
\begin{equation} \label{eq:3.20}
\irn h(u_n)(u_n-u) \,dx \to 0 \quad \hbox{as} \ n \to +\infty.
\end{equation} 

Now by \ef{eq:3.11}, \ef{eq:3.17} and \ef{eq:3.20}, one finds that
\begin{align*}
&\frac{\de_u J(\theta_n, u_n)[u_n-u]}{e^{N\theta_n}} \\
&=\irn \phi \left(\frac{u_n^2+e^{-2\theta_n}|\n u_n|^2}{2}\right)
\Big\{ u_n(u_n -u)+e^{-2\theta_n}\n u_n \cdot \n (u_n - u) \Big\} \,dx \\
&\quad +m\irn u_n (u_n -u) \,dx +\irn h(u_n)(u_n -u) \,dx 
-\irn g_1(u_n)(u_n -u) \,dx \\
&=\irn \phi \left(\frac{u_n^2+e^{-2\theta_n}|\n u_n|^2}{2}\right)
\Big\{ u_n(u_n -u)+e^{-2\theta_n}\n u_n \cdot \n (u_n - u) \Big\} \,dx \\
&\quad +m\irn u_n (u_n -u) \,dx +o_n(1).
\end{align*}
Thus from \ef{eq:3.13} and since $\de_u J(\theta_n, u_n)\to 0$ in $(\Hr)'$, we deduce that
\begin{align*}
& e^{N\theta_n}\irn \phi \left(\frac{u_n^2+e^{-2\theta_n}|\n u_n|^2}{2}\right)
\Big\{ u_n(u_n -u)+e^{-2\theta_n}\n u_n \cdot \n (u_n - u) \Big\} \,dx \\
&\quad +me^{N\theta_n}\irn u_n (u_n -u) \,dx =o_n(1).
\end{align*}
Hence denoting $v_n=u_n(e^{-\theta_n}\cdot)$ and $z_n=u(e^{-\theta_n}\cdot)$, we obtain 
\begin{equation*} 
\Psi'_m(v_n)[v_n-z_n]\to 0 \quad \hbox{as} \ n \to +\infty.
\end{equation*}
Since $\theta_n \to 0$ as $n\to +\infty$, one has $z_n \to u$ strongly in $\H$ 
because $z_n\weakto u$ weakly in $H^1(\RN)$ and $\|z_n\|\to \|u\|$.
Thus by the boundedness of $\{u_n\}$ (and so also of $\{v_n\}$), we get 
\begin{equation} \label{eq:3.21}
\Psi'_m(v_n)[v_n-u]\to 0 \quad \hbox{as} \ n \to +\infty.
\end{equation}
Then from \ef{eq:3.14} and \ef{eq:3.21}, one can apply Lemma \ref{lem:2.2} to obtain
\begin{equation} \label{eq:3.22}
\n v_n(x)\to \n u (x) \quad \hbox{as} \ n \to +\infty \ \hbox{ a.e. in }\RN.
\end{equation}
Moreover arguing as in \cite[Lemma 3.2]{JR}
or \cite[Proposition 3.1]{Stu}, one has from \eqref{phi2} that
\begin{align*}
&\Psi_m(u)-\Psi_m(v_n)-\Psi_m'(v_n)[u-v_n] \\
&\ge \irn \phi \left( \frac{v_n^2+|\n v_n|^2}{2} \right) 
\Big\{ (u^2+|\n u|^2)^{\frac{1}{2}} (v_n^2+|\n v_n|^2)^{\frac{1}{2}}
-(uv_n + \n u \cdot \n v_n ) \Big\} \,dx \\
&\quad +\frac{m}{2} \irn (u-v_n)^2 \,dx 
%\\&
\ge 0.
\end{align*}
Thus from \ef{eq:3.21}, we infer that
\begin{equation} \label{eq:3.23}
\limsup_{n \to +\infty} \Psi_m(v_n)\le \Psi_m(u).
\end{equation}
From \ef{eq:3.14}, \ef{eq:3.22} and \ef{eq:3.23}, 
we can apply Lemma \ref{lem:2.3} to conclude that $v_n \to u$ strongly in $\Hr$. 
Hence, using again the fact that $\theta_n \to 0$ as $n\to +\infty$, 
we also have $u_n \to u$ strongly in $\Hr$.
Using \ef{eq:3.11} again, one finds that $I'(u)=0$
and hence $u$ a solution of \eqref{eq:1} by Proposition \ref{prop:3.1}. 
Finally by \eqref{eq:3.12}, passing to the limit, we can see that
\[
\frac{1}{N} \irn \quas |\n u|^2 \,dx =c>0. \]
Then from \eqref{phi1}, it follows that $u$ is non-trivial.
\end{proof}

\section{Regularity and Positivity of solutions for \ef{eq:1}}\label{se4}

In this section, we show that any solution of \ef{eq:1} 
has $C^{1,\sigma}$-regularity and
any non-negative nontrivial solution of \ef{eq:1} is indeed positive everywhere
under \eqref{phi1} and the slightly stronger assumption \eqref{phi3}.

For this purpose, we first observe that 
\ef{eq:1} can be written by the form:
\beq \label{eq:4.1}
\Div {\bm A}(u, \nabla u) + B(u,\nabla u) =0,
\eeq
where for $(u,\bp) \in \R \times \R^N$,
\begin{align*}
{\bm A}(u, \bp) &= \quasi \bp, \quad
A_i(u, \bp) = \quasi p_i, \ i =1,\cdots, N, \\
B(u,\bp) &= - \quasi u +g(u). 
\end{align*}
Then from \eqref{phi1}, one finds that 
\begin{align} 
{\bm A} (u,\bp) \cdot \bp 
&= \quasi | \bp|^2 \ge \phi_0 | \bp|^2, \label{eq:4.2} \\
| \bm{A} (u,\bp) | 
&\le \quasi |\bp| \le \phi_1 | \bp|, \label{eq:4.3}
\end{align}
\begin{equation*}
-\phi_1 |u| +g(u) \le B(u,\bp) \le \phi_1 |u| + g(u).
\end{equation*}
Especially for $|u| \le M$ and $\bp \in \R^N$, 
it follows from \ef{g2} that
\begin{equation} \label{eq:4.4}
| B(u,\bp)| \le K |u| \quad
\hbox{for some} \ K>0 \ \hbox{depending on} \ M.
\end{equation}
Moreover direct calculations yield that
\begin{align*}
\frac{\partial A_i}{\partial p_j} (u,\bp) \xi_i \xi_j
&= \left\{ \phi' \left( \frac{u^2+|\bp|^2}{2} \right) p_ip_j 
+\quasi \delta_{ij} \right\} \xi_i \xi_j, \quad i,j=1,\cdots , N, \\
\frac{\partial A_i}{\partial u} (u,\bp) 
&= \phi' \left( \frac{u^2+|\bp|^2}{2} \right) u p_i.
\end{align*}
From \eqref{phi3}, for $|u| \le M$, one can see that
\begin{align} \label{eq:4.5}
\sum_{i,j=1}^N \frac{\partial A_i}{\partial p_j} \xi_i \xi_j
&= \phi' \left( \frac{u^2+|\bp|^2}{2} \right) ( \bp \cdot \bm{\xi})^2
+\quasi | \bm{\xi}|^2  \\
&\le \left\{ \left| \phi' \left( \frac{u^2+|\bp|^2}{2} \right) \right| |\bp|^2
+\phi \left( \frac{u^2+|\bp|^2}{2} \right) \right\} |\bm{\xi}|^2 \notag \\
&\le \left\{ 2 \left( \frac{u^2+|\bp|^2}{2} \right) 
\left| \phi' \left( \frac{u^2+|\bp|^2}{2} \right)  \right| +\phi_1 \right\} 
|\bm{\xi}|^2 \notag \\
&\le C |\bm{\xi}|^2, \notag
\end{align}
and, for some $K >0$ depending on $M$,
\begin{align} \label{eq:4.6}
&\sum_{i=1}^N \left( \left| \frac{\partial A_i}{\partial u} \right| +|A_i| \right)
(1+|\bp|) + |B|  \\
&\le N \left| \phi' \left( \frac{u^2+|\bp|^2}{2} \right) \right| 
|u| |\bp| (1+|\bp|) +N \quasi |\bp| (1+|\bp|) \notag \\
&\quad +\quasi |u| +|g(u)| \notag \\
&\le N \left( \frac{u^2+|\bp|^2}{2} \right) 
\left| \phi' \left( \frac{u^2+|\bp|^2}{2} \right) \right| (1+|\bp|)
+N\phi_1 (1+|\bp|)^2 +\phi_1 |u|+|g(u)| \notag \\
&\le K(1+|\bp|)^2 . \notag
\end{align}
Moreover if $(u,\bp) \in \R \times \R^N$ satisfies 
$\phi' \left( \frac{u^2+|\bp|^2}{2} \right) \le 0$, 
we have from \eqref{phi3} that 
\begin{align*} 
\sum_{i,j=1}^N \frac{\partial A_i}{\partial p_j} \xi_i \xi_j
&= \phi' \left( \frac{u^2+|\bp|^2}{2} \right) (\bp \cdot \bm{\xi})^2 
+\quasi |\bm{\xi}|^2 \notag \\
&\ge \phi' \left( \frac{u^2+|\bp|^2}{2} \right) |\bp|^2 |\bm{\xi}|^2
+\quasi |\bm{\xi}|^2 \notag \\
&\ge \left\{ 2 \left( \frac{u^2+|\bp|^2}{2} \right) \phi' \left( \frac{u^2+|\bp|^2}{2} \right)
+\quasi \right\} |\bm{\xi}|^2 \notag \\
&\ge \phi_0 |\bm{\xi}|^2.
\end{align*}
In the case $\phi' \left( \frac{u^2+|\bp|^2}{2} \right) \ge 0$, one easily finds that 
\[
\sum_{i,j=1}^N \frac{\partial A_i}{\partial p_j} \xi_i \xi_j
\ge \phi_0 |\bm{\xi}|^2, \]
yielding that 
%there exists $C>0$ such that
\begin{equation} \label{eq:4.7}
\sum_{i,j=1}^N \frac{\partial A_i}{\partial p_j} \xi_i \xi_j
\ge \phi_0 |\bm{\xi}|^2
\quad \hbox{for all} \ (u,\bp, \bm{\xi}) \in \R \times \R^N \times \R^N.
\end{equation} 

Under these preparations, we establish the following regularity result.

\begin{proposition} \label{prop:4.1}
Assume \eqref{phi1}, \eqref{phi3} and \ef{g1}-\ef{g3}.
Then any weak solution $u \in H^1(\R^N)$ of \ef{eq:1} belongs to the class
$C^{1,\sigma}(\R^N)$ for some $\sigma \in (0,1)$.
\end{proposition}

\begin{proof}
The proof consists of three steps.

\noindent
{\bf Step 1}: $u \in L^q(\R^N)$ for any $q > \frac{2N}{N-2}$.

This kind of property can be obtained by applying the Brezis-Kato lemma
as in \cite[P. 329]{BL}. 
However since our problem \ef{eq:1} is quasilinear, 
we give the proof for the sake of completeness.

For $L>0$ and $s \ge 0$, we define
\[
\varphi:= u \min \{ |u|^{2s} ,L^2 \} \in H^1(\R^N). \]
Then multiplying \ef{eq:1} by $\varphi$, one has
\begin{align} \label{eq:4.8}
&\irn \quas |\n u|^2 \min \{ |u|^{2s},L^2 \} \,dx \\
&+\frac{s}{2} \int_{ \{ x; |u(x)|^s \le L \}}
\quas \big| \n |u|^2 \big|^2 |u|^{2s-2} \,dx \notag \\
&+\irn \quas u \varphi \,dx = \irn g(u) \varphi \,dx. \notag
\end{align}
By \ef{g1}-\ef{g3}, for any $\eps \in (0,m+\phi_0)$, there exists $C_{\eps}>0$ such that
\[
g(s) \le -(m-\eps) s + C_{\eps} s^{\frac{N+2}{N-2}} \quad \hbox{for} \ s \ge 0. \]
Since $u$ and $\varphi$ have the same sign, it holds that
\[
g(u) \varphi \le -(m-\eps) u \varphi +C_{\eps} |u|^{\frac{4}{N-2}} u \varphi. \]
Then from \ef{eq:4.8} and \eqref{phi1}, we get 
\begin{equation} \label{eq:4.9}
\phi_0 \irn \big| \nabla \big( u \min \{ |u|^s,L\} \big) \big|^2 \,dx
\le C \irn |u|^{\frac{4}{N-2}} u \varphi \,dx.
\end{equation}
For any $K>0$, by the H\"older and the Sobolev inequalities, one has
\begin{align*}
\irn |u|^{\frac{4}{N-2}} u \varphi \,dx 
&= \int_{\{ x; |u(x)| \ge K \}} |u|^{\frac{4}{N-2}} u \varphi \,dx
+\int_{\{ x; |u(x)| \le K \}} |u|^{\frac{4}{N-2}} u \varphi \,dx \\
&\le C \left( \int_{\{ x; |u(x)| \ge K \}} |u|^{\frac{2N}{N-2}} \,dx \right)^{\frac{2}{N}}
\left( \irn \big| u \min \{ |u|^s,L \} \big|^{\frac{2N}{N-2}} \,dx \right)^{\frac{N-2}{N}} \\
&\quad + K^{\frac{4}{N-2}} \irn |u|^2 \min \{ |u|^{2s},L^2 \} \,dx \\
&\le C \left( \int_{\{ x; |u(x)| \ge K \}} |u|^{\frac{2N}{N-2}} \,dx \right)^{\frac{2}{N}} 
\irn \big| \nabla ( u \min \{ |u|^s,L\}) \big|^2 \,dx \\
&\quad +K^{\frac{4}{N-2}} \irn |u|^{2s+2} \,dx.
\end{align*}
Since $u \in L^{\frac{2N}{N-2}}(\RN)$, it follows that
\[
\int_{\{ x; |u(x)| \ge K \}} |u|^{\frac{2N}{N-2}} \,dx \to 0 \quad \hbox{as} \ K \to +\infty.\]
Thus from \ef{eq:4.9}, choosing sufficiently large $K$, we find that
\[
\irn \big| \nabla ( u \min \{ |u|^s,L \} ) \big|^2 \,dx
\le C \| u\|_{L^{2s+2}(\RN)}^{2s+2}. \]
Letting $L \to +\infty$, we conclude that
\begin{equation} \label{eq:4.10}
\big\| \nabla (|u|^{s+1}) \big\|_{L^2(\RN)} \le C \| u \|_{L^{2s+2}(\RN)}^{s+1}.
\end{equation}
Now putting $s_0=0$ and $s_i+1=(s_{i-1}+1)\frac{N}{N-2}$ for $i \ge 1$,
one deduces from \ef{eq:4.10} that
\begin{align*}
u \in L^2 \underset{s_0=0}{\Rightarrow} \nabla |u| \in L^2 &\Rightarrow u \in L^{\frac{2N}{N-2}} 
\underset{s_1=\frac{2}{N-2}}{\Rightarrow} \nabla |u|^{\frac{N}{N-2}} \in L^2 \\
&\Rightarrow u \in L^{2 \left( \frac{N}{N-2} \right)^2} 
\underset{s_2=\left( \frac{N}{N-2} \right)^2-1}{\Rightarrow} 
\nabla |u|^{\left( \frac{N}{N-2} \right)^2} \in L^2 \ \cdots
\end{align*}
and hence 
$u \in L^q(\RN)$ for any $q> \frac{2N}{N-2}$, as desired.

\noindent
{\bf Step 2}: $u \in L^{\infty}_{\rm loc}(\RN)$.

By \eqref{phi1} and \ef{g1}-\ef{g3}, one has
\begin{align*}
B(u,\bp) \,{\rm sign} \,u 
&= - \quasi u \,{\rm sign} \,u +g(u) \,{\rm sign} \,u \\
&\le -(\phi_0 +m-\eps)|u| +C_{\eps} |u|^{\frac{N+2}{N-2}} \\
&\le C_{\eps} |u|^{\frac{4}{N-2}} |u|. 
\end{align*}
Thus by Step 1, we obtain
\begin{equation} \label{eq:4.11}
B(u,\bp) \, {\rm sign} \, u \le a(x) |u|, \quad
a(x)=C_{\eps} |u|^{\frac{4}{N-2}} \in L^r(\RN) \ \hbox{for any} \ r>\frac{N}{2}.
\end{equation}
Then the claim follows from \ef{eq:4.2}, \ef{eq:4.11} and
by the Moser type iteration. 
(See \cite{DT} or \cite[Theorem 7.1, P. 286]{LU}.) 

\noindent
{\bf Step 3}: $u\in C^{1,\sigma}(\RN)$ for some $\sigma \in (0,1)$.

Once we get the $L^{\infty}$-boundedness, together with 
\ef{eq:4.4}, \ef{eq:4.5}, \ef{eq:4.6} and \ef{eq:4.7}, 
we are able to apply the regularity result 
for quasilinear elliptic problems of the divergence form \ef{eq:4.1}
due to \cite[Chapter 4, Theorems 3.1, 5.2 and 6.2]{LU}, \cite{Tol}
to conclude that $u\in C^{1,\sigma}(\RN)$ for some $\sigma \in (0,1)$.
\end{proof}

From \ef{eq:4.2}, \ef{eq:4.3} and \ef{eq:4.4}, we can also apply the
strong maximum principle \cite[Theorem 2.5.1]{PuS}
or the Harnack inequality \cite{Tru}.
Then we obtain the following positivity result.

\begin{proposition} \label{prop:4.2}
Assume \eqref{phi1}, \eqref{phi3} and \ef{g1}-\ef{g3}.
Then any non-negative non-trivial (regular) solution of \ef{eq:1}
is positive on $\RN$.
\end{proposition}

\section{Existence of a positive solution and a radial ground state solution}\label{se5}

In this section, we prove the existence of a positive solution
and a radial ground state solution of \ef{eq:1}.

\begin{proof}[Proof of Theorem \ref{main2}]
Under \eqref{phi1} and \eqref{phi3}, we know that
any non-negative non-trivial of \ef{eq:1} is of the class $C^{1,\sigma}$ and
positive on $\RN$ by Propositions \ref{prop:4.1} and \ref{prop:4.2}. 
Thus the claim follows from Theorem \ref{main1}.
\end{proof}

In the next lemma, we show that each solution of \eqref{eq:1} 
satisfies a Pohozaev type identity.

\begin{lemma} \label{lem:5.1}
Assume that \eqref{phi1}, \eqref{phi3} and \ef{g1}-\ef{g3}.
Then if $u\in \H$ is a solution of \eqref{eq:1}, 
it satisfies the following type Pohozaev identity:
\begin{equation} \label{eq:5.1}
\irn \Quas \,dx
-\frac 1N \irn \quas |\n u|^2 \,dx 
-\irn G(u) \,dx =0.
\end{equation}
\end{lemma}

\begin{proof}
We argue as in \cite{PS}. 
By Proposition \ref{prop:4.1}, 
we know that $u \in C^{1,\sigma}(\RN)$ for some $\sigma \in (0,1)$.
Then, since the function 
$$
{\mathcal L}(s,\bp)=\Phi \left( \frac{s^2+|\bp|^2}{2} \right) $$ 
associated with the differential operator in \eqref{eq:1} is strictly convex for all $s \in \R$, 
we can apply the Pohozaev identity due to \cite{DMS}
by choosing $h(x)=h_k(x)=H(x/k)x \in C_0^1(B_{2k}(0),\R^N)$ for $k \in \N$,
where $H\in C_0^1(\R^N)$ is such that $H(x)=1$ on $|x| \le 1$
and $H(x)=0$ for $|x| \ge 2$.
Letting $k \to +\infty$ and taking into account that 
\[
\Phi \left(\frac{u^2+|\n u|^2}{2}\right)\!,\
\phi \left(\frac{u^2+|\n u|^2}{2}\right)|\n u|^2 \hbox{ and } G(u) \in L^1(\RN),
\]
we obtain \eqref{eq:5.1} as claimed. 
\end{proof}

Next we show the existence of a radial ground state solution of \ef{eq:1}.
Let us define by $\CS$ the set of the nontrivial radial solutions of \eqref{eq:1}, namely
\[
\CS:=\{u\in \Hr\setminus \{0\} \ ; \ I'(u)=0\}.
\]
By Theorem \ref{main1}, we know that $\CS\ne \emptyset$. 

\begin{lemma} \label{lem:5.2}
Assume that \eqref{phi1}, \eqref{phi3} and \ef{g1}-\ef{g4}.
Then it holds that
\begin{equation*} 
\inf_{u\in \CS}\|u\|>0.
\end{equation*}
\end{lemma}

\begin{proof}
If $u\in \CS$, since $I'(u)[u]=0$, we have
\[
\irn \phi \left(\frac{u^2+|\n u|^2}{2}\right)\left(u^2+|\n u|^2\right) \,dx
+\irn g_2(u)u \,dx -\irn g_1(u)u \,dx =0.
\]
Therefore, by \eqref{eq:3.5}, \ef{eq:3.6} and \eqref{phi1}, we have
\begin{align*}
\phi_0 \|\n u\|^2_{L^2} +(\phi_0 +m)\|u\|^2_{L^2} 
&\le \irn \phi \left(\frac{u^2+|\n u|^2}{2}\right)\left(u^2+|\n u|^2\right) \,dx
+\irn g_2(u)u \,dx \\
&= \irn g_1(u)u \,dx \\
&\le \e \|u\|_{L^2}^2 + C_\e\|u\|_{L^{2^*}}^{2^*}.
\end{align*}
by which, taking $\e>0$ sufficiently so such that 
$\phi_0+m-\e>0$, the conclusion follows immediately.
\end{proof}

\begin{lemma} \label{lem:5.3}
Assume that \eqref{phi1}, \eqref{phi3} and \ef{g1}-\ef{g3}.
Then it follows that
\begin{equation*}
m_{0,\rm rad}:=\inf_{u\in \CS}I(u)>0.
\end{equation*}
\end{lemma}

\begin{proof}
Suppose by contradiction that $m_{0,\rm rad}=0$. 
Then there exists $\{u_n\} \subset \CS$ such that $I(u_n)\to 0$ as $n \to +\infty$. 
Since $u_n$ is a solution of \ef{eq:1}, 
it follows by Lemma \ref{lem:5.1} that
$u_n$ satisfies the following Pohozaev type identity: 
\begin{align*}
&\irn \Phi \left(\frac{u_n^2+|\n u_n|^2}{2}\right) \,dx
-\frac 1N\irn \phi \left(\frac{u_n^2+|\n u_n|^2}{2}\right)|\n u_n|^2 \,dx \\
&\quad +\irn G_2(u_n) \,dx - \irn G_1(u_n) \,dx =0
\end{align*}
and hence
\[
I(u_n)
=\frac 1N\irn \phi \left(\frac{u_n^2+|\n u_n|^2}{2}\right)|\n u_n|^2 \,dx \to 0 
\quad \hbox{as} \ n \to +\infty.
\]
By \eqref{phi1}, this implies that
\begin{equation} \label{eq:5.2}
\|\n u_n\|_{L^2} \to 0 \quad \hbox{as} \ n \to +\infty.
\end{equation} 
Moreover using \eqref{eq:3.7}, \ef{eq:3.8} \eqref{eq:5.2}, \eqref{phi1} 
and the embedding of $\D$ into $L^{2^*}(\RN)$, we have
\begin{align*}
\frac{ \phi_0+m}{2}\|u_n\|_{L^2}^2
&\le \irn \Phi \left(\frac{u_n^2+|\n u_n|^2}{2}\right) \,dx +\irn G_2(u_n) \,dx \\
&=\frac 1N \irn \phi \left(\frac{u_n^2+|\n u_n|^2}{2}\right)|\n u_n|^2 \,dx
+ \irn G_1(u_n) \,dx \\
& \le o_n(1) +\frac \e 2\|u_n\|^2_{L^2} +\frac{C_\e}{2^*}\|u_n\|_{L^{2^*}}^{2^*} \\
& \le o_n(1) +\frac \e 2\|u_n\|^2_{L^2}.
\end{align*}
Choosing sufficiently small $\eps >0$, together with \eqref{eq:4.2}, 
we find that $\|u_n\|\to 0$ reaching a contradiction with Lemma \ref{lem:5.2}.
\end{proof}

By Lemma \ref{lem:5.3}, we are ready to prove the existence of a radial 
ground state solution of \ef{eq:1}.

\begin{proof}[Proof of Theorem \ref{main3}] 
Let $\{u_n\} \subset \CS$ be a minimizing sequence
such that $I'(u_n)=0$ and $I(u_n) \to m_{0,\rm rad}$ as $n \to +\infty$.
Repeating the arguments of the proof of Theorem \ref{main1}, 
we can prove that $\{u_n\}$ is bounded in $\H$. 
This implies the existence of $\bar u\in \Hr$ such that $u_n \weakto \bar u$ in $\H$.
Arguing as in the proof of Theorem \ref{main1}, 
we deduce that $u_n \to \bar u$ strongly in $\H$, 
and therefore $\bar u$ satisfies
\[
I(\bar u)= m_{0,\rm rad}= \min_{u\in \CS}I(u),
\]
namely, $\bar u$ is a radial ground state solution of \eqref{eq:1}. 
The regularity and the positivity of $\bar u$ 
follow by Propositions \ref{prop:4.1} and \ref{prop:4.2}.
\end{proof}

Finally by applying the Pohozaev identity \ef{eq:5.1}, 
we establish the following non-existence result,
which indicates \ef{g3} and \ef{g4} are {\it almost} optimal.

\begin{theorem} \label{thm:5.4}
Assume \eqref{phi1} and \eqref{phi3}
and $\Phi(t^2) \ge t^2 \phi(t^2)$ on $\R$. 
Then \ef{eq:1} has no non-trivial regular solution
if one of the following conditions holds:
\begin{enumerate}
\item[\rm(i)] $g(s)=-ms+|s|^{p-1}s$, with $m\in (-\phi_0,+\infty)$ and $p \ge \frac{N+2}{N-2}$;

\item[\rm(ii)] $G(s) \le \frac{\phi_0}{2} s^2$, for all $s\in \R$.
\end{enumerate}
\end{theorem}

\begin{proof}
(i). \ Let $u$ be a solution of \eqref{eq:1}. From $I'(u)[u]=0$, it holds that
\[
\irn \left\{ \quas (u^2+|\n u|^2) +mu^2 \right\} \,dx = \irn |u|^{p+1} \,dx. \]
Combining this equation with \ef{eq:5.1}, one finds that
\begin{align} \label{eq:5.3}
&\irn \left\{ \Quas - \left( \frac{u^2+|\n u|^2}{2} \right) \quas
+\frac{1}{N} \quas u^2 + \frac{m}{N} u^2 \right\} \,dx \\
&=\left( \frac{1}{p+1}-\frac{N-2}{2N} \right) \irn |u|^{p+1} \,dx. \notag
\end{align}
Since $p \ge \frac{N+2}{N-2}$, r.h.s. of \ef{eq:5.3} is non-positive.
On the other hand, 
we have from \eqref{phi1} and by the assumption $\Phi(t^2) \ge t^2 \phi(t^2)$ that
\[
\hbox{l.h.s. of \ef{eq:5.3}} \ \ge \frac{\phi_0+m}{N} \irn u^2 \,dx. \]
Since $0< m+\phi_0$, 
l.h.s. of \ef{eq:5.3} is positive if $u \not\equiv 0$.
This is a contradiction and hence $u \equiv 0$.

(ii). \ Let, again, $u$ be a solution of \eqref{eq:1}. Using \ef{eq:5.1} again, one finds that
\begin{equation} \label{eq:5.4}
 \irn \left\{ \Quas - \frac{1}{N} \quas |\n u|^2 - \frac{\phi_0}{2} u^2 \right\} \,dx \\
= \irn \left( G(u)-\frac{\phi_0}{2}u^2 \right) \,dx.  
\end{equation}
By the assumption, r.h.s. of \ef{eq:5.4} is non-positive,
while \eqref{phi1} and the fact that $\Phi(t^2) \ge t^2 \phi(t^2)$ yield that
\begin{align*}
&\Quas - \frac{1}{N} \quas |\n u|^2 - \frac{\phi_0}{2} u^2 \\
&\ge \left( \frac{u^2+|\n u|^2}{2} \right) \quas 
- \frac{1}{N} \quas |\n u|^2 - \frac{\phi_0}{2} u^2 \\
&= \frac{N-2}{2N} \quas |\n u|^2 +\frac{u^2}{2} \left( \quas - \phi_0 \right) \\
&\ge \frac{(N-2)\phi_0}{2N} | \n u |^2 \quad \hbox{a.e. in} \ \RN,
\end{align*}
showing that l.h.s. of \ef{eq:5.4} is positive if $u \not\equiv 0$.
This is a contradiction again.
\end{proof}

\begin{remark} \label{rem:5.5}
From Theorem \ref{thm:5.4} (ii), 
we see that 
a natural assumption for the existence seems to be
\[
G(\zeta) > \frac{\phi_0}{2} \zeta^2 \quad \hbox{for some} \ \zeta>0 \]
instead of \ef{g4}.
However for the moment, we don't know whether the functional $I$
has the mountain pass geometry as in Lemma \ref{lem:3.2}
under this slightly weaker assumption.
\end{remark}

\section{Existence of a ground state solution}\label{se6}

In this section, we show the existence of a ground state solution of \ef{eq:1}
under more stronger assumption \eqref{phi4}.

For this purpose, we begin with the following lemma.

\begin{lemma} \label{lem:6.1}
Assume \eqref{phi1}, \eqref{phi3} and \eqref{phi4}. 
\begin{enumerate}
\item[\rm(i)] $J_1(s,b):= \quasb b^2$ is increasing and convex
with respect to $b$ for all $(s,b) \in \R \times \R_+$.

\item[\rm(ii)] $J_2(s,b):= \Quasb - \frac{1}{N} \quasb b^2$
is increasing and convex with respect to $b$ for all $(s,b) \in \R \times \R_+$.
\end{enumerate}
\end{lemma}

\begin{proof}
(i). \ A direct calculation shows that
\begin{align*}
\frac{\partial J_1}{\partial b} 
&= 2 \quasb b + \quasbd b^3, \\
\frac{\partial^2 J_1}{\partial b^2}
&= 2 \quasb + 5\quasbd b^2 + \quasbdd b^4.
\end{align*}
Since $\phi' \le 0$ by \eqref{phi'neg}, 
we have from \eqref{phi1} and \eqref{phi3} that
\begin{align*}
\frac{\partial J_1}{\partial b} 
&\ge 2b \left\{ \quasb + \left( \frac{s^2+b^2}{2} \right) \quasbd \right\} \\
&\ge 2b \left\{ \frac{\phi_0}{2} +\frac 12 \quasb 
+ \left( \frac{s^2+b^2}{2} \right) \quasbd \right\} \\
&\ge 2\phi_0 b>0.
\end{align*}
Moreover from \eqref{phi4}, 
we deduce the following pointwise estimate:
\begin{equation*}
\frac{1}{2} \frac{\partial^2 J_1}{\partial b^2} \ge
\begin{cases}
\quasb + 5 \left( \frac{s^2+b^2}{2} \right) \quasbd  
\ge 0 & \hbox{if} \ \phi'' \ge 0, \\
\quasb + 5 \left( \frac{s^2+b^2}{2} \right) \quasbd
+2 \left( \frac{s^2+b^2}{2} \right)^2 \quasbdd 
\ge 0 & \hbox{if} \ \phi'' \le 0.
\end{cases}
\end{equation*}

(ii). \ First from $\phi' \le 0$ by \eqref{phi'neg} , we observe that
\[
\frac{\partial J_2}{\partial b} 
= \frac{N-2}{N} \quasb b - \frac{1}{N} \quasbd b^3 
\ge \frac{N-2}{N} \phi_0 b >0. \]
Next, by a simple computation, one has
\begin{align*}
\frac{\partial^2 J_2}{\partial b^2} 
&= \frac{N-2}{N} \quasb + \frac{N-5}{N} \quasbd b^2 -\frac{1}{N} \quasbdd b^4 \\
&= \frac{N-2}{N} \left\{ \quasb + \quasbd b^2 \right\}
- \frac{b^2}{N} \left\{ 3 \quasbd + \quasbdd b^2 \right\} \\
&\ge \frac{N-2}{N} \left\{ \quasb + 2 \left( \frac{s^2+b^2}{2} \right) \quasbd \right\} \\
&\quad - \frac{b^2}{N} \left\{ 3 \quasbd + 2 \left( \frac{s^2+b^2}{2} \right) 
\left| \quasbdd \right| \right\}.
\end{align*}
Then by \eqref{phi3} and \eqref{phi4}, it holds that $\frac{\partial^2 J_2}{\partial b^2} \ge 0$.
\end{proof}

Since $J_2(s,b) > J_2(s,0)= \Phi ( \frac{s^2}{2} )$ for all $b>0$ by Lemma \ref{lem:6.1}, 
we find that
\begin{equation} \label{eq:6.1}
\irn \left\{\Quas - \frac{1}{N} \quas |\n u|^2 - \Phi \left( \frac{u^2}{2} \right) \right\} \,dx >0
\quad \hbox{for any} \ u \in \H \setminus \{ 0\}.
\end{equation}

Next for $u \in \H$, we define
\begin{align*}
P(u) &:= \irn \Quas \,dx - \irn \frac{1}{N} \quas |\n u|^2 \, dx- \irn G(u) \,dx, \\
\P &:= \big\{ u \in \H \setminus \{ 0 \} \ ; \ P(u)=0 \big\}.
\end{align*}
From \eqref{phi1}, \eqref{phi3} and \ef{g1}-\eqref{g3}, $P$ is a $C^1$-functional on $\H$.
The equation $P(u)=0$ is exactly the Pohozaev identity 
as established in Lemma \ref{lem:5.1}. 
Especially by Theorem \ref{main2}, it follows that $\P \ne \emptyset$.
The next result shows that the set $\P$ is actually a $C^1$-manifold.

\begin{proposition} \label{prop:6.2}
Assume \eqref{phi1}, \eqref{phi3}, \eqref{phi4} and \ef{g1}-\ef{g4}.
Then the set $\P$ is a co-dimension one manifold,
bounded away from zero.
Moreover $\P$ is a natural constraint for the functional $I$.
\end{proposition}

\begin{proof} 
The proof consists of four steps.

\smallskip
\noindent
{\bf Step 1}: \ $\P$ is bounded away from zero.

For this purpose, we put $\eps = \frac{N-2}{2N}(\phi_0+m)$
and $c_0 = \min \left\{ \frac{\phi_0+m}{2}, \phi_0 \right\}$.
Then by \ef{eq:3.7}, \ef{eq:3.8}, \eqref{phi1} and \eqref{stima},
one has
\begin{align*}
P(u) &= \frac{N-2}{N}\irn \Quas \,dx 
+\frac{2}{N} \irn \left\{ \Quas - \quas \frac{| \n u|^2}{2} \right\} \,dx \\
&\quad + \irn G_2(u) \,dx - \irn G_1(u) \,dx \\
&\ge \frac{N-2}{2N} \phi_0 \| u\|^2 
+\frac{2}{N} \irn \left\{ \Quas - \left( \frac{u^2+| \n u|^2}{2} \right) \quas\right\} \,dx \\
&\quad 
+\frac{m}{2} \| u\|_{L^2}^2 - \frac{\eps}{2} \| u\|_{L^2}^2 -C \| u \|_{L^{2^*}}^{2^*} \\
&\ge \frac{N-2}{2N} \phi_0 \| \n u \|_{L^2}^2
+\frac{N-2}{4N}(\phi_0 +m) \| u\|_{L^2}^2 - C \| u \|^{2^*} \\
&\ge \frac{N-2}{2N} c_0 \| u\|^2 -C \| u \|^{2^*}
\end{align*}
for some $C>0$.
This implies that there exists $\delta>0$ such that $P(u)>0$ 
for any $u \in \H$ with $0< \| u \| < \delta$
and hence $\P$ is bounded away from zero.

\smallskip
\noindent
{\bf Step 2}: \ if $P'(u)=0$, then $u$ satisfies 
another Pohozaev type identity $\tilde{P}(u)=0$, where
\begin{align} \label{eq:6.2}
\tilde{P}(u) &:= 
\irn \Quas \,dx - \frac{2N-2}{N^2} \irn \quas | \n u|^2 \,dx \\ 
&\quad + \frac{1}{N^2} \irn \quasd | \n u|^4 \,dx -\irn G(u) \,dx. \notag
\end{align}
We note that $\quasd | \n u |^4 \in L^1(\RN)$ because $s|\phi'(s)| \le C$ for $s\in [0,+\infty)$ by \eqref{phi3}.

By a direct calculation, $P'(u)=0$ implies that
$u$ is a weak solution of the following elliptic equation of the divergence form:
\begin{equation} \label{eq:6.3}
\Div \tilde{\bm{A}} (u,\n u) + \tilde{B}(u,\n u)=0,
\end{equation} 
where for $(u,\bp) \in \R \times \RN$,
\begin{align*}
\tilde{\bm{A}}(u,\bp)
&= \frac{N-2}{N} \quasi \bp - \frac{1}{N} \quasid | \bp|^2 \bp, \quad
\tilde{\bm A} = ( \tilde{A}_i)_{i=1,\cdots,N}, \\
\tilde{B}(u,\bp) &= - \quasi u + \frac{1}{N} \quasid | \bp|^2 u +g(u).
\end{align*}
If we could establish that $u \in C^1(\RN)$, 
then we are able to apply the generalized Pohozaev identity due to \cite{DMS}
as in Lemma \ref{lem:5.1}, completing the proof of Step 2.
Thus it remains to show that any weak solution of \ef{eq:6.3} belongs to
the class $C^1(\RN)$ as in Proposition \ref{prop:4.1}.

To this aim, we investigate uniform ellipticity of the operator
$\Div \tilde{\bm A}+ \tilde{B}$.
First from $\phi' \le 0$, \eqref{phi1} and \eqref{phi3}, one finds that
\begin{align} 
\tilde{\bm A}(u,\bp) \cdot \bp
&= \frac{N-2}{N} \quasi | \bp |^2 - \frac{1}{N} \quasid | \bp |^4 
\ge \frac{N-2}{N} \phi_0 | \bp |^2, \notag %\label{eq:6.4} 
\\
| \tilde{\bm A}(u,\bp) | 
&\le \frac{N-2}{N} \phi_1 | \bp | + \frac{2}{N} C | \bp |,\nonumber %\label{eq:6.5} 
\\
| \tilde{B}(u,\bp) | 
&\le K |u| \quad \hbox{for} \ |u| \le M. \nonumber%\label{eq:6.6}
\end{align}
Next, by a direct computation, we have
\begin{align*}
\sum_{i,j=1}^N \frac{\partial \tilde{A}_i}{\partial p_j} \xi_i \xi_j
&= \frac{N-4}{N} \quasid ( \bp \cdot {\bm \xi} )^2
-\frac{1}{N} \quasidd |\bp|^2 (\bp \cdot {\bm \xi})^2 \\
&\quad + \frac{N-2}{N} \quasi |{\bm \xi}|^2 
-\frac{1}{N} \quasid |\bp|^2 |{\bm \xi}|^2, \\
\frac{\partial \tilde{A}_i}{\partial u} (u,\bp)
&= \frac{N-2}{N} \quasid p_i u - \frac{1}{N} \quasidd |\bp|^2 p_i u.
\end{align*}
By using the assumptions $s | \phi'(s) | \le C$, $s^2 | \phi''(s) | \le C$
and \eqref{phi1}, one gets
\begin{equation*}% \label{eq:6.7}
\sum_{i,j=1}^N \frac{\partial \tilde{A}_i}{\partial p_j} \xi_i \xi_j
\le C | {\bm \xi} |^2.
\end{equation*}
Moreover applying the Young inequality, we also have
\begin{align*}
\sum_{i=1}^N \left| \frac{\partial \tilde{A}_i}{\partial u} (u,\bp) \right|
&\le (N-2) \left| \quasid \right| |\bp| |u| 
+ \left| \quasidd \right| |\bp|^3 |u| \\
&\le (N-2) \left( \frac{u^2 + |\bp|^2}{2} \right) \left| \quasid \right|
+\left| \quasidd \right| \left( \frac{3}{4} | \bp|^4 + \frac{1}{4} |u|^4 \right) \\
&\le (N-2) \left( \frac{u^2 + |\bp|^2}{2} \right) \left| \quasid \right|
+3 \left( \frac{u^2 + |\bp|^2}{2} \right)^2 \left| \quasidd \right|,
\end{align*}
from which we conclude that
\begin{equation*}% \label{eq:6.8}
\sum_{i=1}^N \left( \left| \frac{\partial \tilde{A}_i}{\partial u} \right|
+| \tilde{A}_i | \right) (1+|\bp|) +|\tilde{B} |
\le C(1+|\bp|)^2 +K|u| \le C(1+|\bp|)^2
\quad \hbox{for} \ |u| \le M.
\end{equation*}
Moreover by $\phi' \le 0$ by \eqref{phi'neg}, \eqref{phi3} and \eqref{phi4}, it follows that
\begin{align*}% \label{eq:6.9}
&\sum_{i,j=1}^N \frac{\partial \tilde{A}_i}{\partial p_j} \xi_i \xi_j \\
&= \frac{1}{N} \left\{ (N-2) \quasi |{\bm \xi}|^2
+(N-1) \quasid (\bp \cdot {\bm \xi})^2 
- \quasid |\bp|^2 |{\bm \xi}|^2 \right\} \notag \\
&\quad -\frac{1}{N} \left\{
3 \quasid + \quasidd |\bp|^2 \right\} (\bp \cdot {\bm \xi})^2 \notag \\
&\ge \frac{N-2}{N} \left\{ \quasi 
+2 \left( \frac{u^2 + |\bp|^2}{2} \right) \quasid \right\} |{\bm \xi}|^2 \notag \\
&\quad -\frac{1}{N} \left\{ 3 \quasid + 2 \left( \frac{u^2 + |\bp|^2}{2} \right) 
\left| \quasidd \right| \right\} (\bp \cdot {\bm \xi})^2\notag  \\
&\ge \frac{N-2}{N} \phi_0 |{\bm \xi}|^2. \notag 
\end{align*}
Finally using $\phi' \le 0$, \eqref{phi1} and \ef{g1}-\ef{g3}, we obtain
\begin{equation*} %\label{eq:6.10}
\tilde{B}(u,\bp) \,{\rm sign} \,u
\le -( \phi_0+m-\eps) |u| + C_{\eps} |u|^{\frac{N+2}{N-2}}
\le C_{\eps} |u|^{\frac{4}{N-2}} |u|.
\end{equation*}
Under these preparations, we are able to apply the regularity result
as in Proposition \ref{prop:4.1} to obtain $u \in C^{1,\sigma}(\RN)$ 
for some $\sigma \in (0,1)$.

\smallskip
\noindent
{\bf Step 3}: \ $\P$ is a co-dimension one manifold.

For this purpose, we argue as in \cite[Lemma 1.4]{Sha}
and suppose by contradiction that there exists $u \in \P$ such that $P'(u)=0$.
Then by using two Pohozaev type identities \ef{eq:5.1} and \ef{eq:6.2},
we obtain
\[
\irn \left\{ (N-2) \quas |\n u|^2  - \quasd |\n u|^4 \right\} \,dx =0. \]
Since $\phi' \le 0$, it follows by \eqref{phi1} that
\begin{align*}
0 &\ge \irn \quasd | \n u|^4 \,dx
= (N-2) \irn \quas |\n u|^2 \,dx \\
&\ge (N-2)\phi_0 \irn |\n u|^2 \,dx,
\end{align*}
yielding that $u \equiv 0$. 
This contradicts to Step 1 
and hence $P'(u) \ne 0$ for any $u \in \P$.

\smallskip
\noindent
{\bf Step 4}: \ $\P$ is a natural constraint for $I$.

Again, we follow the argument in \cite[Theorem 1.6]{Sha} (see also \cite{PoSe}).
Let $u \in \P$ be a critical point of the functional $I|_{\P}$.
By Step 3, we are able to apply the method of Lagrange multiplier
to obtain the existence of $\mu \in \R$ such that
\[
I'(u)=\mu P'(u). \]
As a consequence, together with Step 2, $u$ satisfies the following identity:
\begin{align*}
P(u) &=\mu \tilde{P}(u) \\
&= \mu \left[ P(u) - \frac{1}{N^2} \irn \left\{ (N-2) \quas |\n u|^2
-\quasd | \n u|^4 \right\} \,dx \right].
\end{align*}
Since $P(u)=0$, this yields that
\[
\mu \irn \left\{ (N-2) \quas |\n u|^2 -\quasd | \n u|^4 \right\} \,dx =0. \]
However as we can see by the proof of Step 3, 
this is possible only if $\mu =0$.
This completes the proof.
\end{proof}

Next, let us denote $m_0$ by the ground state energy level:
\begin{equation*}
m_0 := \inf_{u \in S} I(u), \quad \hbox{where} \
S: = \{ u \in \H \setminus \{ 0\} \ ; \ I'(u)=0 \}. 
\end{equation*}
By Theorem \ref{main1}, it holds that $S \ne \emptyset$.
Moreover since the proofs of Lemmas \ref{lem:5.2} and \ref{lem:5.3}
do not rely on the radial symmetry, 
one can see that $m_0>0$.

\begin{lemma} \label{lem:6.3}
Let 
\[
b := \inf \{ I(u) \ ; \ u \in \P \}. \]
We have that $0< b\le m_0$. Moreover if $b$ is attained, then it holds that $m_0=b$.
\end{lemma}

\begin{proof}
First by Proposition \ref{prop:6.2}, we find that
\[
I(u) = \frac{1}{N} \irn \quas |\n u|^2 \,dx \ge \frac{\phi_0}{N} \| \n u \|_{L^2}^2>0, \]
for any $u \in \P$, and hence $b> 0$.

If $I'(u)=0$ and $u \not \equiv 0$, 
Proposition \ref{prop:4.1} and Lemma \ref{lem:5.1} yield that $u \in \P$ and thus
\[
b \le I(u) \quad \hbox{for any} \ u \in S.\]
This implies that $b \le m_0$.

On the other hand if $b$ is attained, then there exists $u \in \P$ such that it is a minimizer of $I|_\P$. 
Then by Proposition \ref{prop:6.2}, we have that
$I'(u)=0$ and $u \not \equiv 0$, from which one concludes that $m_0 \le I(u)=b$.
Thus we obtain $m_0=b$, as claimed.
\end{proof}

\begin{lemma} \label{lem:6.4}
Assume \eqref{phi1}, \eqref{phi3}, \eqref{phi4} and \ef{g1}-\ef{g4}.
For any $u \in \H \setminus \{ 0 \}$ which satisfies
$\irn G(u)-\Phi ( \frac{u^2}{2} ) \,dx>0$, 
there exists $\theta_0>0$ such that 
$u_{\theta_0}(\cdot) = u (\cdot/\theta_0 ) \in \P$.
If further $P(u) \le 0$, then it holds that $0< \theta_0 \le 1$.
\end{lemma}

\begin{proof}
First we define a $C^1$-function $f(\theta)$ on $[0,+\infty)$ by
\[
f(\theta) := I(u_{\theta}) = \theta^N \irn 
\Phi \left( \frac{u^2+\theta^{-2} | \n u|^2}{2} \right) \,dx
-\theta^N \irn G(u) \,dx. \]
Since $\irn G(u)-\Phi ( \frac{u^2}{2} ) \,dx >0$, 
it follows that $f(\theta) \to -\infty$ as $\theta \to +\infty$.
Moreover \eqref{phi1} and \ef{g1}-\ef{g3} imply that
$f(\theta)>0$ for sufficiently small $\theta>0$.
Thus there exists $\theta_0>0$ such that $f'(\theta_0)=0$.
Then by a direct calculation, one finds that $P(u_{\theta_0})=0$.

Next we suppose that $P(u) \le 0$. 
Then from \ef{eq:6.1}, we have
\begin{align*}
0&< \irn \left\{ \Quas - \frac{1}{N} \quas |\n u|^2 - \Phi \left( \frac{u^2}{2} \right) \right\} \,dx \\
&\le \irn G(u)- \Phi \left( \frac{u^2}{2} \right) \,dx, 
\end{align*}
from which we obtain the existence of $\theta_0>0$ so that
$P(u_{\theta_0})=0$.
Now since $P(u_{\theta_0})=0$ and $P(u) \le 0$, one finds that
\begin{align} 
\irn \left\{ \Phi \left( \frac{u^2+\theta_0^{-2} | \n u|^2}{2} \right)
- \frac{1}{N} \phi \left( \frac{u^2+\theta_0^{-2} | \n u|^2}{2} \right)
\theta_0^{-2} | \n u|^2 \right\} \,dx 
&= \irn G(u) \,dx, \label{eq:6.11} \\
\irn \left\{ \Quas - \frac{1}{N} \quas |\n u|^2 \right\} \,dx
&\le \irn G(u) \,dx. \label{eq:6.12}
\end{align}
From \ef{eq:6.11} and \ef{eq:6.12}, it follows that
\[
\irn J_2(u, \theta_0^{-1} | \n u| ) \,dx 
\ge \irn J_2(u, | \n u|) \,dx. \]
Then by Lemma \ref{lem:6.1} (ii), we conclude that $\theta_0 \le 1$.
\end{proof}

Under these preparations, we are ready to prove Theorem \ref{main4}.

\begin{proof}[Proof of Theorem \ref{main4}]
We argue as in \cite{AW, CJS}.
By Lemma \ref{lem:6.3}, it suffices to show that 
there exists $u_0 \in \P$ such that
\[
I(u_0)= b = \min_{u \in \P} I(u). \]
Let $\{ u_n \} \subset \P$ be a minimizing sequence for $b$.
We may assume that $u_n \ge 0$ because $P(|u_n|)=P(u_n)=0$
and $I(|u_n|)=I(u_n)$. 

For all $n\ge 1$, let $u_n^*$ be the Schwarz symmetrization of $u_n $.
Then, by Lemma \ref{lem:6.1}, we are able to apply the 
generalized Polya-Szeg\"{o} inequality (see \cite[Proposition 3.11]{HK}) to obtain
\begin{equation} \label{polya}
\irn J_i ( u_n^*, | \n u_n^*|) \,dx \le \irn J_i( u_n, |\n u_n |) \,dx,
\qquad\hbox{ for }i=1,2. 
\end{equation}
Applying \ef{polya} for $i=2$, since $\irn G(u_n^*) \,dx = \irn G(u_n) \,dx$, 
we find that $P(u_n)=0$ implies $P(u_n^*) \le 0$. 
Thus by Lemma \ref{lem:6.4}, 
there exists $0< \theta_n \le 1$ such that 
$v_n:= (u_n^*)_{\theta_n} = u_n^* ( \cdot/\theta_n ) \in \P$.

Now applying the generalized 
Polya-Szeg\"{o} inequality \eqref{polya} for $i=1$, we find that
\begin{align*}
b+o_n(1) = I(u_n) &= \frac{1}{N} \irn \quasn |\n u_n|^2 \,dx \\
&\ge \frac{1}{N} \irn 
\phi \left( \frac{ (u_n^*)^2 +| \n u_n^*|^2}{2} \right) | \n u_n^*|^2 \,dx \\
&= \frac{ \theta_n^{-N+2}}{N} \irn \phi
\left( \frac{ v_n^2 +\theta_n^2 | \n v_n|^2}{2} \right) | \n v_n|^2 \,dx.
\end{align*}
Since $v_n \in \P$, $\phi' \le 0$ by \eqref{phi'neg} and $\theta_n \le 1$, one obtains
\[
b+o_n(1) \ge \frac{1}{N} \irn \phi \left( \frac{ v_n^2 +| \n v_n|^2}{2} \right) | \n v_n|^2 \,dx
=I(v_n) \ge b, \]
yielding that $\{ v_n \}$ is also a minimizing sequence for $b$.
By the radial symmetry of $v_n$, 
we can argue as in the proof of Theorem \ref{main1} to prove that
$v_n \to u_0$ in $\H$ for some $u_0 \in \P$.
This implies that 
\begin{equation} \label{eq:6.13}
I(u_0)=b=m_0= \min_{u \in S} I(u), 
\end{equation}
as claimed.
Moreover since $v_n \in H_r^1(\RN)$ and $v_n \ge 0$, 
it follows that $u_0$ is radially symmetric with respect to the origin
(up to translation) and non-negative. 
Then by Propositions \ref{prop:4.1} and \ref{prop:4.2}, 
we have $u_0 \in C^{1,\sigma}(\RN)$ for some $\sigma \in (0,1)$ and positive on $\RN$,
finishing the proof.
\if0
Finally, we show that any ground state solution of \ef{eq:1} is 
of the class $C^{1,\sigma}(\RN)$ for some $\sigma \in (0,1)$, 
radially symmetric with respect to some point and with fixed sign on $\RN$.
To this aim, let $w \in \H$ be a ground state solution of \ef{eq:1}.
The regularity of $w$ follows by Proposition \ref{prop:4.1}.
Moreover as we have observed above, 
one has 
\[
0=P(w)=P(|w|) \quad \hbox{and} \quad m_0=I(w)=I(|w|), \]
from which we may assume that $w \ge 0$.
Then by Proposition \ref{prop:4.2}, it follows that $w>0$ in $\RN$.
Finally letting $w^*$ be the Schwarz symmetrization of $w$ and
arguing similarly as above, 
we can see that $\tilde{w}:= w^*(\cdot/\theta_0 )$ for some $0< \theta_0 \le 1$
satisfies $\tilde{w} \in \P$ and
\[
m_0= I(w) \ge I(\tilde{w}) \ge m_0. \]
This implies that $w$ can be assumed to be radially symmetric, with respect to some point.
finishing the proof. \fi
\end{proof}

\begin{remark} \label{rem:6.5}
By the variational characterization \ef{eq:6.13}, the oddness of $g$ in \ef{g1}
and Proposition \ref{prop:4.2}, 
we can see that any ground state solution $w$ of \ef{eq:1} has fixed sign on $\RN$.

One also expects that any ground state solution $w$ of \ef{eq:1} 
is radially symmetric with respect to some point as in \cite{AW, CJS}.
But for the moment, we are not able to prove it.
Actually for this purpose, we first need that the function $J_1(s,b)$ defined in 
Lemma \ref{lem:6.1} is strictly convex with respect to $b$,
which follows by assuming that
\[
0<\phi(s)+5s \phi'(s) -2s^2 |\phi''(s)| \quad \hbox{for all} \ s\in [0,+\infty). \]
Then one can apply the case of equality for the generalized Polya-Szeg\"o inequality to $J_1$
{\rm(}\cite[Theorem 2.11 and Corollary 2.12]{HK}{\rm)}, 
showing that $w=w^*$ a.e. in $\RN$ provided that
\begin{equation} \label{eq:6.14}
{\mathcal L} \{ x \in \RN \ ; \ 0<w^*(x)< {\rm ess\,sup} \, w \quad 
\hbox{and} \quad \nabla w^*(x)=0 \} =0.
\end{equation}
{\rm(}See also \cite[Theorem 1.1]{BZ} and \cite[Corollary 2.33]{Ka}.{\rm)}
Especially if $w$ is analytic, then \ef{eq:6.14} can be established.
We also note that one cannot apply the symmetry result due to \cite{M}
because of the loss of a variational characterization like \ef{eq:4} in our problem.
\end{remark}

\subsection*{Acknowledgment}
The first author is supported by PRIN 2017JPCAPN 
{\em Qualitative and quantitative aspects of nonlinear PDEs}.
The second author is supported by JSPS Grant-in-Aid for Scientific Research (C) (No. 18K03383).

\end{document}